\documentclass[a4paper,11pt,twoside,notitlepage]{amsart}

\usepackage{amsmath,amssymb,amsthm,mathtools,mathrsfs}
\usepackage[shortlabels]{enumitem}
\usepackage[ocgcolorlinks,linkcolor=blue,citecolor=blue,urlcolor=blue]{hyperref}
\usepackage{url}
\numberwithin{equation}{section}
\mathtoolsset{showonlyrefs}
\allowdisplaybreaks

\newtheorem{theorem}{Theorem}[section]
\newtheorem{proposition}[theorem]{Proposition}
\newtheorem{lemma}[theorem]{Lemma}
\newtheorem{corollary}[theorem]{Corollary}
\newtheorem{definition}[theorem]{Definition}
\newtheorem{remark}[theorem]{Remark}
\newtheorem{question}{Question}[section]

\DeclareMathOperator{\tr}{tr}

\DeclareMathOperator{\diam}{diam}
\DeclareMathOperator{\dist}{dist}

\DeclareMathOperator{\Sym}{Sym}
\DeclareMathOperator{\Gr}{Gr}

\newcommand{\R}{\mathbb R}

\newcommand{\p}{\partial}
\newcommand{\eps}{\varepsilon}
\newcommand{\ol}{\overline}
\newcommand{\Hq}{\mathcal H^q}
\newcommand{\Rq}{\mathcal R_q}
\newcommand{\Gk}{\Gamma_k}
\newcommand{\para}[1]{\vspace{3mm}\noindent\textbf{#1.}}

\title[Affine section tomography for $k$-Hessian equations]{Affine section tomography for inverse source problems in $k$-Hessian equations with restricted large boundary data}

\author[Y.-H. Lin]{Yi-Hsuan Lin}
\address{Department of Applied Mathematics, National Yang Ming Chiao Tung University, Hsinchu, Taiwan \& Fakult\"at f\"ur Mathematik, University of Duisburg-Essen, Essen, Germany}
\email{yihsuanlin3@gmail.com}


\begin{document}
	
	\begin{abstract}
		We consider an inverse source problem for the $k$-Hessian equation
		\begin{equation*}
			\sigma_k(D^2u)=f(x)
		\end{equation*}
		on the $k$-admissible branch in a smooth uniformly convex domain in $\R^n$, where $2\le k\le n$. We prove that the nonlinear Dirichlet-to-Neumann map determines the positive smooth source from its values on the restricted large-data rays $t\phi_E|_{\p\Omega}$, where $E\in\Gr(k-1,n)$, $\phi_E(x)=|P_Ex|^2/2$, and $t$ is sufficiently large. The Hessian of each boundary profile has exactly $k-1$ large directions and is flat on $V=E^\perp$. Thus, the leading profile lies on a rank $k-1$ face of the $k$-Hessian structure, and the first source-dependent correction is governed by the missing directions. More precisely, this correction solves fiberwise Poisson equations on the affine sections $\Omega\cap(y+V)$ of dimension $q=n-k+1$. We prove that these sectionwise solutions patch smoothly through glancing points where the sections collapse, and we obtain boundary normal derivative asymptotics by local barriers. The boundary flux of the correction gives the section integrals $\int_{\Omega\cap(y+V)}f\,d\mathcal H^q$. Varying $E$ yields the affine $q$-plane Radon transform of the zero extension of $f$. We give an explicit reconstruction formula through the Fourier slice identity, and the injectivity of the affine Radon transform gives uniqueness. The endpoint $k=n$ recovers the Monge--Amp\`ere chord/X-ray geometry, while the range $n\ge3$ and $2\le k<n$ gives inverse source results for genuinely non-determinant Hessian equations.
	\end{abstract}

	\maketitle
	\tableofcontents
	
	\section{Introduction}
	
	The $k$-Hessian equation is one of the standard fully nonlinear elliptic partial differential equations (PDEs) associated with the Hessian matrix. It interpolates between the Poisson equation and the Monge--Amp\`ere equation. When $k=1$, the operator is the trace of the Hessian: $\sigma_1(D^2u)=\tr(D^2u)$; while when $k=n$, the operator is the determinant of the Hessian: $\sigma_n(D^2u)=\det D^2u$. Thus, the endpoint $k=n$ is precisely the Monge--Amp\`ere equation on the convex, or equivalently $n$-admissible, branch. In dimension two, the $k$-Hessian equations are only the Poisson equation and the Monge--Amp\`ere equation; since the present theorem starts at $k=2$, its two-dimensional case is exactly the Monge--Amp\`ere endpoint. The genuinely intermediate non-determinant Hessian equations occur when $n\ge3$ and $2\le k\le n-1$. They arise in nonlinear potential theory, in the theory of $k$-convex functions and Hessian measures, and in geometric problems where elementary symmetric functions of curvature-type quantities are prescribed. We refer to \cite{Wang2009_kHessian} for an account of the $k$-Hessian equation, to \cite{ChouWang2001} for its variational theory, and to \cite{TrudingerWang1999} for Hessian measures associated with $k$-convex functions.
	
	We use the following notation for the equation. Let $\Omega\subset\R^n$ be a bounded smooth domain. For a symmetric matrix $M$, let $\lambda(M)=(\lambda_1(M),\ldots,\lambda_n(M))$ be its eigenvalues. We write $\sigma_j(M)$ for the $j$-th elementary symmetric function of $\lambda(M)$:
	\begin{equation}\label{eq:sigma-def-intro}
		\sigma_j(M)=\sum_{1\le i_1<\cdots<i_j\le n}\lambda_{i_1}(M)\cdots\lambda_{i_j}(M).
	\end{equation}
	The $k$-Hessian equation is
	\begin{equation}\label{eq:kHessian}
		\sigma_k(D^2u)=f(x) \quad \text{in } \Omega.
	\end{equation}
	This equation is elliptic only after one chooses the correct branch. The relevant branch is the $k$-admissible cone
	\begin{equation}\label{eq:Gk-intro}
		\Gk=\left\{\lambda\in\R^n: \, \sigma_j(\lambda)>0 \text{ for } 1\le j\le k\right\}.
	\end{equation}
	We also write $M\in\Gk$ if $\lambda(M)\in\Gk$. A $C^2$ function $u$ is called $k$-admissible if $D^2u(x)\in\Gk$ for all $x\in\Omega$. On this branch, the operator is elliptic, and $\sigma_k^{1/k}$ is concave. This admissible structure is fundamental in the solvability theory of Caffarelli--Nirenberg--Spruck \cite{CNS_nonlinear_Hessian}.
	
	In the model \eqref{eq:kHessian}, the unknown function $u$ may be viewed as a potential, a height function, or a state variable, depending on the context. The quantity $f$ is the prescribed interior density for the $k$-Hessian measure of $u$. In the linear endpoint $k=1$, this is the usual source density in the Poisson equation, such as a charge density or mass density in potential theory. In the Monge--Amp\`ere endpoint $k=n$, it is the prescribed determinant density, or volume distortion density, associated with the gradient map. For intermediate $k$, the source prescribes the $k$-th elementary Hessian response of the potential. Thus, $f$ is the interior quantity one would like to determine from boundary observations.
	
	The present paper studies whether this source can be recovered from boundary measurements. For a smooth boundary value $g$, let $u_g$ be the smooth $k$-admissible solution of
	\begin{equation}\label{eq:forward-intro}
		\begin{cases}
			\sigma_k(D^2u_g)=f & \text{in } \Omega,\\
			u_g=g & \text{on } \p\Omega.
		\end{cases}
	\end{equation}
	The nonlinear Dirichlet-to-Neumann (DN) map is
	\begin{equation}\label{eq:DN-intro}
		\Lambda_f:C^\infty(\p\Omega)\to C^\infty(\p\Omega), \quad
		g \mapsto \p_\nu u_g|_{\p\Omega},
	\end{equation}
	where $\nu$ is the outward unit normal. Thus, the boundary measurement consists of prescribing the boundary potential or height $g$ and recording the normal response $\p_\nu u_g$. Depending on the model, this normal derivative may be interpreted as a flux, a boundary slope, or a boundary response.
	
	\begin{question}[Inverse source problem]\label{Q:IP}
		Can we determine the unknown source $f$ from the nonlinear DN map $\Lambda_f$?
	\end{question}
	
	This is an inverse source problem: the unknown is an interior density, while the available measurements are made only at the boundary. For a linear elliptic equation, the knowledge of the DN map with arbitrary Dirichlet data is not enough to recover an arbitrary source. Here, one uses the nonlinear boundary response for a family of boundary values. We show that for the $k$-Hessian equation with $k\ge2$, a suitably large boundary family reveals an integral transform of the source.

	\para{Earlier literature}
	There is now a substantial literature on inverse problems for nonlinear partial differential equations. A common approach is to linearize the nonlinear DN map and reduce the problem to an inverse problem for the linearized equation. This method goes back at least to the work of Isakov \cite{Isa93}. Later developments showed that higher-order linearizations can reveal nonlinear information which is not visible at first order; see, for instance, \cite{KN02, Sun96, Sun10}. A systematic use of nonlinear interactions was introduced for geometric nonlinear problems in \cite{KLU2018} and subsequently developed for semilinear elliptic equations in \cite{FO20, LLLS2019nonlinear}. We also refer to \cite{LLLS2019partial, LLST2022inverse, KU20b, KU20a, FLL23} for related semilinear elliptic inverse problems, including results with partial data and fractional power type nonlinearities.
	
	Quasilinear elliptic inverse problems have also been studied in several settings, including \cite{CFKKU2021calderon, KKU2022partial, CNV2019reconstruction, LW24_quasi}. Inverse problems related to the minimal surface equation and related geometric quasilinear models were considered, for example, in \cite{ABN20, CLLT24, CLT24, Nur24}. For a broader account of nonlinear inverse problems and further references, see the survey \cite{lassas2025introduction}.
	
	The present paper is closest in spirit to inverse source problems for nonlinear equations. In the linear case, the source cannot generally be recovered from the DN map because of the usual compactly supported gauge obstruction. Nonlinear equations may break this obstruction. Recent results in this direction include inverse source problems for semilinear elliptic and parabolic equations \cite{LL24_elliptic_source, KLL_reaction_source}, and a quasilinear elliptic inverse source problem \cite{liimatainen2026inverse}. In the fully nonlinear setting, inverse source problems for the Monge--Amp\`ere equation and for more general two-dimensional admissible fully nonlinear equations were studied in \cite{LL2025IP_Monge_Ampere, LLW26_fully_nonlinear}. These works use mechanisms tied either to the Monge--Amp\`ere structure or to two-dimensional linearization phenomena.
	
	The large boundary data approach of \cite{CG26_large} gives another mechanism for the Monge--Amp\`ere equation. The large cylindrical profile leaves one missing direction, and the first correction is governed by an equation on one-dimensional chords; the resulting boundary information is of X-ray type. The present work develops a different Hessian-specific degeneracy. A rank $k-1$ large profile leaves the whole missing space $V=E^\perp$, whose dimension is $q=n-k+1$. The first variation identity
	\begin{equation}\label{eq:first_variation_id}
		D\sigma_k(P_E)[H]=\tr(P_VH)
	\end{equation}
	then turns the leading correction into a Poisson problem on the affine sections $\Omega\cap(y+V)$. For $2\le k<n$, these sections are genuinely higher-dimensional. Thus, the inverse problem is reduced not to chord integrals but to the affine $q$-plane Radon transform of the source. The proof also requires a new regularity step for sectionwise Poisson solutions through glancing collapse and a barrier argument converting the large-data expansion into boundary normal derivative asymptotics. In this way, the intermediate Hessian equations give genuinely non-determinant fully nonlinear inverse source results in dimensions $n\ge3$.

	\para{Difficulties and ideas of the proof}
	Two points distinguish this problem from the usual nonlinear inverse source arguments. The first one is the source obstruction already present in the linear equation $\Delta u=f$: changing $f$ by $\Delta w$ with $w$ compactly supported in the interior does not change the boundary Cauchy data. Thus, one needs to use the nonlinearity essentially. A common nonlinear strategy is to use the full DN map on an open set of boundary data and differentiate it several times. This is not the strategy here. The data used below are only large one-parameter rays for each direction $E$, so the proof cannot rely on higher-order linearization in arbitrary boundary directions.
	
	The second point is that the useful large profiles lie close to a degenerate rank $k-1$ face of the admissible structure. The imposed Hessian has $k-1$ large directions, while the operator $\sigma_k$ requires $k$ directions to produce its first nonzero contribution. This rank defect is precisely what reveals the source, but it also destroys ellipticity in the full domain at the leading scale. The first correction is therefore not governed by a uniformly elliptic equation in $\Omega$. It is governed by Poisson equations on the lower-dimensional affine sections parallel to $V=E^\perp$.
	
	This creates two analytic issues. First, the fiberwise Poisson solutions are initially defined section by section, and the sections may become tangent to $\p\Omega$ and collapse at glancing points. One has to prove that the sectionwise Dirichlet solutions nevertheless form a smooth function on $\ol\Omega$. This requires a local normal form for uniformly convex domains near glancing points and a regularity result for elliptic Dirichlet problems on shrinking balls. Second, the global comparison argument gives an $L^\infty$ large-data expansion, but the inverse problem needs the boundary normal derivative of the first correction. Since the large solutions are close to a degenerate Hessian profile, one cannot simply differentiate a uniformly elliptic expansion. We instead use local barriers near boundary points where the $V$-sections meet $\p\Omega$ transversely to pass from the $L^\infty$ expansion to the DN asymptotics.
	
	The proof starts from the algebraic identity \eqref{eq:first_variation_id}. If $u_{t,E}^f$ denotes the solution with boundary value $t\phi_E$, this identity forces the first correction in
	\begin{equation*}
		u_{t,E}^f=t\phi_E+t^{1-k}w_E^f+o(t^{1-k})
	\end{equation*}
	to satisfy $\Delta_Vw_E^f=f$ on the affine sections $\Omega\cap(y+V)$ with zero section boundary value. The large-$t$ DN asymptotics determine $\p_\nu w_E^f$ on the boundary points where the $V$-sections meet $\p\Omega$ transversely. The divergence theorem on each section then converts this boundary information into the section integrals
	\begin{equation*}
		\int_{\Omega\cap(y+V)} f\,d\mathcal H^q .
	\end{equation*}
	Thus, one large ray for a fixed $E$ recovers all section integrals in the direction $V=E^\perp$, and varying $E$ gives the full affine $q$-plane Radon transform of the zero extension of the source. This also gives a direct reconstruction formula: the Fourier transform of the zero extension of $f$ is obtained from the Fourier transform, in the section parameter, of the recovered affine section data. We record this formula after the main uniqueness theorem.

	\para{Mathematical formulations} The forward problem in the setting of the theorem is classical. If $\Omega$ is smooth and uniformly convex and $f\in C^\infty(\ol\Omega)$ is positive, the existence and uniqueness results in \cite{CNS_nonlinear_Hessian} give a unique smooth admissible solution for every smooth boundary value. We recall this in Section \ref{sec:forward}, together with the local smooth dependence of the solution on the boundary value. This means that the nonlinear DN map is defined on the usual smooth boundary data class.
	
	We now describe the large boundary family used in the inverse theorem. For $0\le m\le n$, let $\Gr(m,n)$ be the Grassmannian of $m$-dimensional linear subspaces of $\R^n$. Thus, an element $E\in\Gr(k-1,n)$ is a choice of $k-1$ linear directions. Let $P_E$ be the Euclidean orthogonal projection onto $E$, and set
	\begin{equation}\label{eq:V-def-intro}
		V=E^\perp, \quad q=\dim V=n-k+1.
	\end{equation}
	The space $E$ will be the space of large Hessian directions, and $V$ will be the complementary space of missing directions. Associated with $E$, define
	\begin{equation}\label{eq:phi-intro}
		\phi_E(x):=\frac12 |P_E x|^2.
	\end{equation}
	The boundary values used below are
	\begin{equation}\label{eq:large-data-intro}
		g_{t,E}=t\phi_E|_{\p\Omega}, \quad t\gg1.
	\end{equation}
	Since $D^2(t\phi_E)=tP_E$, this profile has exactly $k-1$ large Hessian directions and is flat in the directions in $V$. The operator $\sigma_k$ needs $k$ directions to give a leading contribution. Hence, the rank $k-1$ profile is one direction short, and the first correction is forced to appear in the missing directions $V$.
	
	This rank defect leads to a sectionwise equation. Let $\Omega_E:=P_E(\Omega)\subset E$. For $y\in\Omega_E$, define the affine section
	\begin{equation}\label{eq:sections-intro}
		\Omega_{E,y}=\Omega\cap(y+V).
	\end{equation}
	The variable $y$ labels the section, while the variables inside the section lie in $V$. If $z$ denotes the variable in $V$, then $\Delta_V$ means the Euclidean Laplacian in the $z$ variables only, with $y$ fixed. For a function defined in $\Omega$, the fiberwise Laplacian is given by 
	\begin{equation}\label{eq:DeltaV-intro}
		\Delta_Vu=\tr(P_VD^2u).
	\end{equation}
	On the boundary, the relevant transversality condition is measured by the $V$-component of the outward unit normal. We define
	\begin{equation}\label{eq:non-glancing}
		\Gamma_E:=\left\{x\in\p\Omega:\ |P_V\nu(x)|>0\right\},
		\quad
		\mathcal G_E:=\left\{x\in\p\Omega:\ P_V\nu(x)=0\right\}.
	\end{equation}
	Points in $\Gamma_E$ are called non-glancing points for the family of $V$-sections, while points in $\mathcal G_E$ are called glancing points. Equivalently, $x\in\mathcal G_E$ means that $V\subset T_x\p\Omega$. The boundary normal derivative asymptotics will be obtained on compact subsets of $\Gamma_E$, and the smoothness of the fiberwise correction through $\mathcal G_E$ will be proved in Section \ref{sec:Poisson_corrections}.
	
	Let $u_{t,E}^f$ be the solution with boundary value $t\phi_E$. The leading term of $u_{t,E}^f$ as $t\to\infty$ is the imposed quadratic profile $t\phi_E$. The next nontrivial term is the first term in the expansion that detects the source $f$. We write this expansion as
	\begin{equation}\label{eq:intro-asymp}
		u_{t,E}^f=t\phi_E+t^{1-k}w_E^f+o(t^{1-k}).
	\end{equation}
	Therefore, $w_E^f$ is the coefficient of the first lower-order term after the large profile $t\phi_E$. This correction is not the first linearized solution obtained by differentiating the solution map with respect to small boundary perturbations. Rather, it is the first asymptotic correction along the large ray $t\phi_E|_{\p\Omega}$, and its boundary flux converts the large-data DN asymptotics into section integrals of the source. This coefficient is determined by a Poisson equation on each nondegenerate section:
	\begin{equation}\label{eq:section-poisson-intro}
		\begin{cases}
			\Delta_Vw_E^f=f & \text{in } \Omega_{E,y},\\
			w_E^f=0 & \text{on } \p\Omega_{E,y}.
		\end{cases}
	\end{equation}
	The boundary normal derivative asymptotics determine the flux of $w_E^f$ through the section boundary. By the divergence theorem, this flux is the integral of $f$ over $\Omega_{E,y}$. It follows that the nonlinear DN map on the large boundary family determines the affine $q$-plane Radon transform of the zero extension of $f$. The Fourier slice identity then gives an explicit reconstruction formula, and the injectivity of the affine Radon transform gives uniqueness.
	
	When $k=n$, one has $\sigma_n(D^2u)=\det D^2u$, and the $n$-admissible branch is the convex branch of the Monge--Amp\`ere equation. In this endpoint, the space $V$ is one-dimensional. Writing $V=\R\omega$, the profile $\phi_E$ becomes $|P_{\omega^\perp}x|^2/2$, and the sections are chords in the direction $\omega$. For $2\le k<n$, the missing space has dimension $q\ge2$, and the proof uses higher-dimensional sections instead of chords. This intermediate case can occur only in dimensions $n\ge3$, and it is the genuinely non-determinant part of the $k$-Hessian result.
	
	The result should also be distinguished from the two-dimensional Monge--Amp\`ere inverse source theorem \cite{LL2025IP_Monge_Ampere} and from the two-dimensional fully nonlinear theorem \cite{LLW26_fully_nonlinear}. In dimension two, the $k$-Hessian family contains only the Laplace and Monge--Amp\`ere equations, so the present intermediate $k$-Hessian mechanism is not a two-dimensional one. The latter work treats more general operators $F(D^2u)$ in the plane and uses higher-order linearization together with structural assumptions on $F$. The present argument is specific to the $k$-Hessian class, but it is non-perturbative in the boundary data and gives higher-dimensional non-determinant examples when $n\ge3$ and $2\le k\le n-1$. The structural input is the first variation of $\sigma_k$ at a rank $k-1$ matrix: 
	\begin{equation}\label{eq:structure-intro}
		D\sigma_k(P_E)[H]=\tr(P_VH).
	\end{equation}
	Here $D\sigma_k(P_E)[H]=\frac{d}{ds}\big|_{s=0}\sigma_k(P_E+sH)$. This identity is responsible for the sectionwise Poisson equation.
	
	Related problems for Hessian equations have also been studied in several directions. Besides the Dirichlet theory \cite{CNS_nonlinear_Hessian}, we refer to \cite{Wang2009_kHessian} for background on $k$-Hessian equations and to \cite{ChouWang2001} for variational aspects. Serrin-type overdetermined problems for Hessian equations were considered in \cite{BNST_08}, while the present paper concerns the recovery of an unknown source from boundary measurements.
	
	The restriction $k\ge2$ is essential. If $k=1$, the equation is $\Delta u=f$. The construction above would give $q=n$, where $q=\dim V$ is defined in \eqref{eq:V-def-intro}, and the corresponding integral transform would only see the total integral over the whole space. This is not injective. There is also a direct PDE obstruction. Let $w\in C^\infty_c(\Omega)$ be nonzero, and let $\widetilde f=f+\Delta w$. After replacing $w$ by a sufficiently small nonzero multiple, we may also assume that $\widetilde f$ is positive. If $u_g$ solves $\Delta u_g=f$ with boundary value $g$, then $\widetilde u_g=u_g+w$ solves $\Delta \widetilde u_g=\widetilde f$ with the same boundary value $g$. Since $w$ is compactly supported in $\Omega$, one has $\p_\nu \widetilde u_g=\p_\nu u_g$ on $\p\Omega$. Thus, the full DN maps for $f$ and $\widetilde f$ agree, although the sources are different when $\Delta w$ is not identically zero. For this reason, the theorem begins at $k=2$.
	
	If $k=n$, then the equation is the Monge--Amp\`ere equation $\det D^2u=f$ on the convex admissible branch. Hence, the theorem below includes the Monge--Amp\`ere inverse source problem in all dimensions. In particular, when $n=2$, the range $2\le k\le n$ contains only $k=2$, so the theorem is exactly the two-dimensional Monge--Amp\`ere endpoint. The genuinely non-determinant cases covered by the theorem are precisely those with $n\ge3$ and $2\le k\le n-1$. 
	
	It is worth emphasizing that the parameter $y$ of the affine sections is not introduced as an additional boundary input. For a fixed plane $E$, the boundary value is only the one-parameter ray $t\phi_E|_{\p\Omega}$, $t\to\infty$. The different section parameters appear because the boundary response is observed pointwise on $\p\Omega$, and the portions $\p\Omega\cap(y+V)$ separate the affine sections. Thus, the large-data limit for a single fixed $E$ recovers all translates of the $q$-plane direction $V=E^\perp$. This is the sense in which the nonlinear boundary measurement produces affine section tomography of the source.
	
	We now state the main result, which answers Question \ref{Q:IP}. For a positive source $f\in C^\infty(\ol\Omega)$, the solution associated with the boundary value $t\phi_E$ will be denoted by $u_{t,E}^f$.
	
	\begin{theorem}\label{thm:main}
		Let $\Omega\subset\R^n$ be a bounded smooth uniformly convex domain, let $2\le k\le n$, and let $f_1,f_2\in C^\infty(\ol\Omega)$ satisfy $f_j\ge c_0>0$. For each $E\in\Gr(k-1,n)$ and each $t>0$, let $u_{j,t,E}$ be the unique smooth $k$-admissible solution of
		\begin{equation}\label{eq:main-forward}
			\begin{cases}
				\sigma_k(D^2u_{j,t,E})=f_j & \text{ in } \Omega, \\
				u_{j,t,E}=t\phi_E   &\text{ on } \p\Omega,
			\end{cases}
		\end{equation}
		for $j=1,2$. Assume that for every $E\in\Gr(k-1,n)$ there exists $t_E>0$ such that
		\begin{equation}\label{eq:main-DN-equality}
			\Lambda_{f_1}\big(t\phi_E|_{\p\Omega}\big)=\Lambda_{f_2}\big(t\phi_E|_{\p\Omega}\big)
			\quad \text{on }\p\Omega, 
		\end{equation}
		for all $t\ge t_E$. Then $f_1=f_2$ in $\Omega$.
	\end{theorem}
	
	Thus, only the values of the nonlinear DN maps on the restricted large-data family 
	\begin{equation*}
		\left\{t\phi_E|_{\p\Omega}: E\in\Gr(k-1,n), t\ge t_E\right\}
	\end{equation*}
	are used. The threshold for the largeness of $t$ may depend on $E$; no single lower bound for $t$ is needed uniformly over all $E\in\Gr(k-1,n)$. The same large-data limits also give the exact reconstruction formula in Corollary \ref{cor:reconstruction}.
	
	\begin{remark}[Size of the data set]\label{rem:size-data-set}
		The measurements used in Theorem \ref{thm:main} are much smaller than the full nonlinear DN map. For each fixed $E\in\Gr(k-1,n)$, the proof only uses the one-parameter large ray $\left\{t\phi_E|_{\p\Omega}: t\ge t_E\right\}$. The full family is obtained by letting $E$ vary in the finite-dimensional Grassmannian $\Gr(k-1,n)$. The measurements form a finite-dimensional large-data family, not an open subset of $C^\infty(\p\Omega)$.
		
		The section parameter $y\in E$ is not prescribed through additional boundary data. Instead, for a fixed $E$, the large-$t$ boundary response determines the boundary normal derivative of the correction $w_E^f$ on the boundary portions $\p\Omega\cap(y+V)$ for all sections $\Omega\cap(y+V)$ with nonempty interior. Integrating this boundary information over $\p\Omega\cap(y+V)$ gives $\int_{\Omega\cap(y+V)} f\,d\mathcal H^q$. 
		
		For sections with empty or degenerate intersections, the corresponding section integral is zero. Hence, one large ray for each $E$ recovers all translates in the corresponding direction $V=E^\perp$, and varying $E$ gives the full affine $q$-plane Radon transform.
	\end{remark}

    We also record the resulting reconstruction formula. Here, the notation is prepared to keep the statement concise. Let $q=n-k+1$. For a positive source $f$ whose restricted large-data DN map is known, define the following quantities. If $E\in\Gr(k-1,n)$, set $V=E^\perp$ and use the non-glancing set $\Gamma_E$ defined in \eqref{eq:non-glancing}. By Lemma \ref{lem:section-geometry}, for every $y\in P_E(\Omega)$ one has
    \[
    \p\Omega\cap(y+V)\subset \Gamma_E.
    \]
    Hence, the large-data DN limit from Corollary \ref{cor:DN-limit} can be used on the boundary of each nondegenerate section. Define
    \begin{equation}\label{eq:reconstruction-NE}
    	N_E^f(x):=\lim_{t\to\infty}
    	t^{k-1}\left(\Lambda_f(t\phi_E|_{\p\Omega})(x)-t\p_\nu\phi_E(x)\right),
    	\quad x\in\Gamma_E.
    \end{equation}
    Then $N_E^f=\p_\nu w_E^f$ on $\Gamma_E$.
    
    For $V\in\Gr(q,n)$, put $E=V^\perp$. Motivated by the section flux identity in Lemma \ref{lem:flux}, define the recovered section data by
    \begin{equation}\label{eq:reconstruction-section-data}
    	\mathcal S_f(V,y):=
    	\int_{\p\Omega\cap(y+V)}
    	|P_V\nu(x)|\,N_E^f(x)\,dS_{q-1}(x),
    	\quad y\in E,
    \end{equation}
    when $y\in P_E(\Omega)$, and set $\mathcal S_f(V,y)=0$ otherwise. The proof below shows that $\mathcal S_f(V,y)$ is the affine section integral of the zero extension of $f$.
    
    Finally, let $F=f\mathbf 1_\Omega$ be the zero extension of $f$ to $\R^n$, where $\mathbf 1_\Omega$ is the characteristic function of $\Omega$. We use the Fourier transform convention
    \begin{equation*}
    	\widehat F(\xi)=\int_{\R^n}e^{-\mathsf{i}x\cdot\xi}F(x)\,dx.
    \end{equation*}
	
	\begin{corollary}[Reconstruction formula]\label{cor:reconstruction}
		Let $\Omega\subset\R^n$ be a bounded smooth uniformly convex domain, let $2\le k\le n$, and let $f\in C^\infty(\ol\Omega)$ satisfy $f\ge c_0>0$. Assume that, for every $E\in\Gr(k-1,n)$, the values $\Lambda_f(t\phi_E|_{\p\Omega})$ are known for all sufficiently large $t$. Then, for each $\xi\in\R^n\setminus\{0\}$ and any choice of $q$-plane $V_\xi\subset\xi^\perp$, one has
		\begin{equation}\label{eq:reconstruction-Fourier}
			\widehat F(\xi)=\int_{V_\xi^\perp}e^{-\mathsf{i}y\cdot\xi}\mathcal S_f(V_\xi,y)\,dy.
		\end{equation}
				The right-hand side is independent of the choice of $V_\xi$. Moreover,
		\begin{equation}\label{eq:reconstruction-inverse-Fourier}
			F=\mathcal F^{-1}\widehat F
			\quad \text{in } \mathcal S'(\R^n),
		\end{equation}
		and hence $f=F|_\Omega$. In other words, for any measurable choice
		$\xi\mapsto V_\xi\in\Gr(q,n)$ with $V_\xi\subset\xi^\perp$ for $\xi\ne0$,
		\begin{equation}\label{eq:reconstruction-pointwise}
			f(x)=(2\pi)^{-n}\int_{\R^n}e^{\mathsf{i}x\cdot\xi}\bigg[\int_{V_\xi^\perp}e^{-\mathsf{i}y\cdot\xi}
			\mathcal S_f(V_\xi,y)\,dy\bigg]\,d\xi,\quad x\in\Omega,
		\end{equation}
		where the value at $\xi=0$ is immaterial and the inverse Fourier transform is understood in the standard distributional sense.
	\end{corollary}
	
	\begin{corollary}[Powers of the $k$-Hessian]\label{cor:power}
		Let $\theta>0$. Consider the equation
		\begin{equation}\label{eq:power-equation}
			\sigma_k(D^2u)^\theta=f(x)
		\end{equation}
		on the $k$-admissible branch. Let $f_1,f_2\in C^\infty(\ol\Omega)$ satisfy $f_j\ge c_0>0$. If the nonlinear DN maps for \eqref{eq:power-equation} agree on the same restricted large-data family as in Theorem \ref{thm:main}, then $f_1=f_2$ in $\Omega$.
	\end{corollary}

	\para{Organization of the paper}
    The paper is organized as follows. Section~\ref{sec:forward} recalls the forward theory, the admissible branch, and the nonlinear DN map. Section~\ref{sec:Poisson_corrections} constructs the fiberwise Poisson correction and proves its smoothness through glancing points. Section~\ref{sec:asymp} proves the large-data asymptotics by global and local barriers. Section~\ref{sec:recovery} recovers the source by the affine Radon transform and proves both the reconstruction formula and the power corollary.

	\section{The forward problem and admissibility}\label{sec:forward}
	
	We recall the basic facts used below. The admissible cone $\Gk$, defined in \eqref{eq:Gk-intro}, is an open convex cone containing the positive orthant $\{\lambda\in\R^n:\lambda_i>0 \text{ for all } i\}$. The ellipticity of $\sigma_k$ on $\Gk$ means that, if $M\in\Gk$, then the Newton tensor
	\begin{equation}\label{eq:newton-tensor}
		T_{k-1}^{ij}(M)=\frac{\p\sigma_k}{\p M_{ij}}(M)
	\end{equation}
	is positive definite. The concavity of $\sigma_k^{1/k}$ on $\Gk$ will only be used through the standard comparison principle.
	
	\begin{lemma}\label{lem:comparison}
		Let $D\subset\R^n$ be a bounded domain. Let $u,v\in C^2(\ol D)$ satisfy $D^2u(x),D^2v(x)\in\Gk$ for every $x\in\ol D$. Suppose that $u\le v$ on $\p D$ and
		\begin{equation}\label{eq:comparison-assumption}
			\sigma_k(D^2u)\ge \sigma_k(D^2v) \quad \text{in } D.
		\end{equation}
		Then $u\le v$ in $D$.
	\end{lemma}
	
	\begin{proof}
		Set $w=u-v$. Since $\Gk$ is convex, for every $x\in\ol D$ and every $0\le s\le1$,
		\[
		M_s(x)=D^2v(x)+s(D^2u(x)-D^2v(x))
		\]
		belongs to $\Gk$. The set $\{M_s(x):x\in\ol D,\ 0\le s\le1\}$ is compact and contained in $\Gk$. Therefore, the Newton tensor is uniformly positive definite along this set.
		
		By the fundamental theorem of calculus,
		\begin{equation}\label{eq:comparison-linearization}
			\sigma_k(D^2u)-\sigma_k(D^2v)=a^{ij}\p_{ij}w,
		\end{equation}
		where $a^{ij}(x)=\int_0^1\frac{\p\sigma_k}{\p M_{ij}}(M_s(x))\,ds$. The matrix $(a^{ij})$ is uniformly positive definite on $\ol D$. By \eqref{eq:comparison-assumption}, $a^{ij}\p_{ij}w\ge0$ in $D$. The weak maximum principle gives
		\[
		\sup_D w\le \sup_{\p D}w\le0.
		\]
		This proves $u\le v$ in $D$.
	\end{proof}
	
	The next result gives the forward solvability used in Theorem \ref{thm:main}.
	
	\begin{proposition}\label{prop:CNS-forward}
		Let $\Omega\subset\R^n$ be a bounded smooth uniformly convex domain, let $2\le k\le n$, and let $f\in C^\infty(\ol\Omega)$ satisfy $f\ge c_0>0$. For every $g\in C^\infty(\p\Omega)$, the Dirichlet problem
		\begin{equation}\label{eq:CNS-forward}
			\begin{cases}
				\sigma_k(D^2u)=f & \text{in } \Omega,\\
				u=g & \text{on } \p\Omega
			\end{cases}
		\end{equation}
		has a unique smooth $k$-admissible solution $u\in C^\infty(\ol\Omega)$. Since $f\ge c_0>0$, the solution satisfies $D^2u(x)\in\Gk$ for every $x\in\ol\Omega$.
		
		Moreover, the solution map is locally smooth at the boundary value. More precisely, if $g_0\in C^\infty(\p\Omega)$ and $u_0$ is the corresponding solution, then for every integer $m\ge2$ and every $\alpha\in(0,1)$ there are neighborhoods $\mathcal U$ of $g_0$ in $C^{m+2,\alpha}(\p\Omega)$ and $\mathcal V$ of $u_0$ in $C^{m+2,\alpha}(\ol\Omega)$ such that $\mathcal U\ni g\mapsto u_g\in\mathcal V$ is a smooth map.
	\end{proposition}
	
	\begin{proof}
		The existence, uniqueness, and boundary smoothness follow from the Dirichlet theory for Hessian equations on uniformly $(k-1)$-convex domains \cite[Theorem 2]{CNS_nonlinear_Hessian}. In the notation of Caffarelli--Nirenberg--Spruck, the boundary condition required for the $k$-Hessian equation is the uniform $(k-1)$-convexity of $\p\Omega$. A uniformly convex domain has all principal curvatures bounded below by a positive constant, and hence it is uniformly $(k-1)$-convex for every $2\le k\le n$. Since $f\in C^\infty(\ol\Omega)$ satisfies $f\ge c_0>0$ and $g\in C^\infty(\p\Omega)$, the theorem gives a unique smooth admissible solution in $\Omega$. Smoothness up to $\p\Omega$ follows from the boundary regularity in the same theorem and the standard bootstrapping for the uniformly elliptic equation along the admissible solution.
		
		We prove the local smooth dependence by the implicit function theorem. Fix $g_0$ and let $u_0$ be the corresponding solution. The solution is smooth up to the boundary and is $k$-admissible in $\Omega$. Since $D^2u_0$ is continuous on $\ol\Omega$, the boundary limiting eigenvalues are contained in $\ol\Gk$. Also, by continuity of the equation,
		\[
		\sigma_k(D^2u_0)=f\ge c_0>0
		\quad \text{on } \ol\Omega.
		\]
		We now use the Newton--Maclaurin inequalities. For completeness, we recall how the needed Newton--Maclaurin inequality follows from the standard concavity of $E_\ell^{1/\ell}$ on $\Gamma_\ell$. Set
		\[
		E_j(\lambda):=\frac{\sigma_j(\lambda)}{\binom{n}{j}}.
		\]
		We first prove the adjacent inequality. Let $2\le \ell\le k$ and define $\Phi_\ell(\lambda):=E_\ell(\lambda)^{1/\ell}$, then the function $\Phi_\ell$ is concave and homogeneous of degree one on $\Gamma_\ell$. For $\lambda\in\Gk\subset\Gamma_\ell$, concavity at $\lambda$ with comparison point $\mathbf 1=(1,\ldots,1)$ gives
		\begin{equation}\label{eq:Phi_ell_inequality}
			1=\Phi_\ell(\mathbf 1)\le \Phi_\ell(\lambda)+D_\lambda\Phi_\ell(\lambda)[\mathbf 1-\lambda].
		\end{equation}
		Since $\Phi_\ell$ is homogeneous of degree one, Euler's identity for homogeneous functions gives $D_\lambda\Phi_\ell(\lambda)[\lambda]=\Phi_\ell(\lambda)$. Indeed, this follows by differentiating
		$\Phi_\ell(s\lambda)=s\Phi_\ell(\lambda)$ with respect to $s$ at $s=1$. Using the algebraic identity $\sum_{i=1}^n\frac{\partial\sigma_\ell(\lambda)}{\partial\lambda_i}=(n-\ell+1)\sigma_{\ell-1}(\lambda)$ and \eqref{eq:Phi_ell_inequality}, we can compute
		\[
		1=\Phi_\ell(\mathbf 1)\le D_\lambda\Phi_\ell(\lambda)[\mathbf 1]=	\frac1\ell\binom{n}{\ell}^{-1/\ell}		\sigma_\ell(\lambda)^{1/\ell-1}(n-\ell+1)\sigma_{\ell-1}(\lambda).
		\]
		Since $\frac{n-\ell+1}{\ell}=\frac{\binom{n}{\ell}}{\binom{n}{\ell-1}}$, this becomes $D_\lambda\Phi_\ell(\lambda)[\mathbf 1]=E_{\ell-1}(\lambda)E_\ell(\lambda)^{-(\ell-1)/\ell}$. Therefore, $E_{\ell-1}(\lambda)\ge E_\ell(\lambda)^{(\ell-1)/\ell}$, or, $E_{\ell-1}(\lambda)^{1/(\ell-1)}\ge E_\ell(\lambda)^{1/\ell}$. Iterating this adjacent inequality gives
		\begin{equation}\label{eq:NM-ineq}
			\bigg(\frac{\sigma_j(\lambda)}{\binom{n}{j}}\bigg)^{1/j}\ge
			\bigg(\frac{\sigma_\ell(\lambda)}{\binom{n}{\ell}}\bigg)^{1/\ell},
			\quad 1\le j<\ell\le k,
		\end{equation}
		for $\lambda\in\Gk$. If $\lambda\in\ol\Gk$, we apply the inequality to $\lambda+s\mathbf 1\in\Gk$ and let $s\downarrow0$. Thus, \eqref{eq:NM-ineq} holds for every $\lambda\in\ol\Gk$.
		
		Taking $\ell=k$ in \eqref{eq:NM-ineq}, we get $\sigma_j(\lambda)\ge \binom{n}{j}\big(\frac{\sigma_k(\lambda)}{\binom{n}{k}}\big)^{j/k}$ for $1\le j\le k-1$. Applying this with $\lambda=\lambda(D^2u_0(x))$ gives $\sigma_j(D^2u_0(x))>0$, $1\le j\le k-1$, $x\in\ol\Omega$, because $\sigma_k(D^2u_0(x))=f(x)\ge c_0>0$. Hence, the eigenvalues of $D^2u_0$ actually remain in $\Gk$ on $\ol\Omega$. By compactness of $\ol\Omega$,
		\begin{equation}\label{eq:compact-admissible-set}
			\{\lambda(D^2u_0(x)):x\in\ol\Omega\}
		\end{equation}
		is contained in a compact subset of $\Gk$.
		
		Let $\gamma:C^{m+2,\alpha}(\ol\Omega)\to C^{m+2,\alpha}(\p\Omega)$ with $\gamma U=U|_{\p\Omega}$ be the trace map. By the standard extension theorem for H\"older spaces on smooth domains, there exists a bounded linear operator $\mathcal E:C^{m+2,\alpha}(\p\Omega)\to C^{m+2,\alpha}(\ol\Omega)$ such that $\gamma(\mathcal E h)=h$, for every $h\in C^{m+2,\alpha}(\p\Omega)$. Thus, $\mathcal E$ is a linear right inverse of the trace map.
		
		For $g$ close to $g_0$, write
		\begin{equation}\label{eq:IFT-parametrization}
			u=u_0+\mathcal E(g-g_0)+w, \quad w\in X,
		\end{equation}
		where $X=\{w\in C^{m+2,\alpha}(\ol\Omega):w|_{\p\Omega}=0\}$. This parametrization fixes the boundary condition, and $u|_{\p\Omega}=g_0+(g-g_0)+0=g$. Moreover, if $u$ has boundary value $g$, then
		\begin{equation}\label{eq:w-zero-trace}
			w=u-u_0-\mathcal E(g-g_0)
		\end{equation}
		has zero trace. Hence, the Dirichlet problem near $u_0$ is equivalent to solving for the zero-trace unknown $w$.
		
		Define
		\begin{equation}\label{eq:IFT-map}
			\mathcal F(w,g)
			=
			\sigma_k(D^2(u_0+\mathcal E(g-g_0)+w))-f.
		\end{equation}
		This is a smooth map from a neighborhood of $(0,g_0)$ in $X\times C^{m+2,\alpha}(\p\Omega)$ to $C^{m,\alpha}(\ol\Omega)$, since $\sigma_k$ is a polynomial in the entries of the Hessian.
		
		The derivative in the $w$ variable at $(0,g_0)$ is
		\begin{equation}\label{eq:linearized-kHessian}
			D_w\mathcal F(0,g_0)v=F^{ij}\p_{ij}v, \quad
			F^{ij}=\frac{\p\sigma_k}{\p M_{ij}}(D^2u_0).
		\end{equation}
		The tensor $(F^{ij})$ is positive definite on the $k$-admissible cone. By \eqref{eq:compact-admissible-set}, it is uniformly positive definite on $\ol\Omega$. Hence, the Dirichlet realization $v\mapsto F^{ij}\p_{ij}v$ with $v|_{\p\Omega}=0$ is an isomorphism from $X$ to $C^{m,\alpha}(\ol\Omega)$ by the maximum principle, the Fredholm alternative, and Schauder theory \cite{gilbarg2015elliptic}.
		
		The implicit function theorem gives a unique local branch $g\mapsto w(g)$, and hence a smooth local solution map $g\mapsto u_g$ in the stated H\"older spaces. The admissibility persists after shrinking $\mathcal U$, since $\Gk$ is open. Applying the same argument in all H\"older scales and using elliptic bootstrapping gives smooth dependence in the $C^\infty$ topology.
	\end{proof}

	\begin{remark}\label{rem:why-smooth-dependence}
		The proof of Theorem \ref{thm:main} does not use higher-order linearization of the nonlinear DN map. It uses the values of $\Lambda_f$ on the large family \eqref{eq:large-data-intro} and then passes to the limit $t\to\infty$. The smooth dependence in Proposition~\ref{prop:CNS-forward} is included to make the nonlinear DN map a smooth boundary response map in the usual H\"older and smooth categories, and to justify the first variation formula recorded below.
	\end{remark}
	
	For a fixed positive source $f$, the nonlinear DN map is therefore well defined by \eqref{eq:DN-intro}.
	In Theorem \ref{thm:main}, only the values on the boundary data family
	\begin{equation}\label{eq:restricted-family}
		\mathcal B_{\rm large}=\left\{t\phi_E|_{\p\Omega}:\, E\in\Gr(k-1,n), \, t\ge t_E\right\}
	\end{equation}
	are used. The threshold $t_E$ may depend on $E$. All asymptotic estimates are proved for a fixed $E$, which is sufficient since the Radon transform is recovered in one direction $V=E^\perp$ at a time.
	
	Although higher-order linearization is not used in the proof of Theorem \ref{thm:main}, Proposition \ref{prop:CNS-forward} allows one to differentiate the solution map. If $u_g$ solves \eqref{eq:CNS-forward} and $h\in C^\infty(\p\Omega)$, the first variation $v_h:=Du_g[h]$ is well defined and satisfies
	\begin{equation}\label{eq:first-var-khessian}
		\begin{cases}
			F^{ij}\p_{ij}v_h=0 & \text{in } \Omega,\\
			v_h=h & \text{on } \p\Omega,
		\end{cases}
	\end{equation}
	where $F^{ij}=\frac{\p\sigma_k}{\p M_{ij}}(D^2u_g)$.

	\section{Fiberwise Poisson corrections}\label{sec:Poisson_corrections}
	
	Fix $E\in\Gr(k-1,n)$. Recall that $V=E^\perp$ and $q=\dim V=n-k+1$. Then every point $x\in\R^n$ has a unique decomposition
	\begin{equation}\label{eq:x-y-z}
		x=y+z, \quad y=P_Ex\in E, \quad z=P_Vx\in V.
	\end{equation}
	We also set $\Omega_E=P_E(\Omega)$, $\Omega_{E,y}=\Omega\cap(y+V)$, for $y\in\Omega_E$. The set $\Omega_E$ parametrizes the sections, and $\Omega_{E,y}$ is the $q$-dimensional section of $\Omega$ with parameter $y$.
	
	The purpose of this section is not merely to solve the Dirichlet problem on each fixed section. For the inverse problem argument, the coefficient $w_E^f$ has to be a smooth function on the original domain, so that its ambient normal derivative on $\p\Omega$ is meaningful and can be compared with the DN asymptotics. Away from glancing points, this is standard parameter-dependent elliptic regularity. The main issue is the behavior near points where $V$ becomes tangent to $\p\Omega$. At such points, the corresponding affine sections shrink to a point, and the Dirichlet problem is posed on a family of collapsing domains. The next lemmas show that uniform convexity gives a quadratic normal form for this collapse and that the sectionwise solutions have smooth expansions through it.
	
	The gradient, Hessian, and Laplacian in the $V$ variables are denoted by $\nabla_V$, $D_V^2$, and $\Delta_V$. After fixing $y\in E$, these are the usual Euclidean differential operators acting on the variable $z\in V$ along the affine plane $y+V$. If $\{e_1,\ldots,e_q\}$ is an orthonormal basis of $V$, then
	\begin{equation}\label{eq:Delta-V-basis}
		\Delta_Vu(y+z)=\sum_{\alpha=1}^q \frac{d^2}{ds^2}\Big|_{s=0}u(y+z+se_\alpha).
	\end{equation}
	Equivalently, in coordinate-free notation,
	\begin{equation}\label{eq:Delta-V}
		\Delta_Vu=\tr(P_VD^2u).
	\end{equation}
	In particular, $\Delta_V$ contains only the second derivatives in the $V$ directions; it does not contain second derivatives in the $E$ directions or mixed $E$-$V$ derivatives. For each fixed parameter $y$, the equations below are ordinary Dirichlet problems for the Euclidean Laplacian on the $q$-dimensional domain $\Omega_{E,y}\subset y+V$.
	
	\begin{definition}\label{def:fiber-correction}
		Let $f\in C^\infty(\ol\Omega)$. For every $y\in\Omega_E$, define $w_E^f(y,\cdot)$ by
		\begin{equation}\label{eq:fiber-w}
			\begin{cases}
				\Delta_V w_E^f(y,\cdot)=f(y+\cdot)& \text{ in } \Omega_{E,y}, \\
				w_E^f(y,\cdot)=0 & \text{ on } \p\Omega_{E,y}.
			\end{cases}
		\end{equation}
		We also define the section barrier $b_E$ by
		\begin{equation}\label{eq:fiber-b}
			\begin{cases}
				\Delta_V b_E(y,\cdot)=1 & \text{ in } \Omega_{E,y}, \\
				b_E(y,\cdot)=0 &\text{ on } \p\Omega_{E,y}.
			\end{cases}
		\end{equation}
	\end{definition}
	
	The functions are initially defined section by section. We use the terminology and notation from \eqref{eq:non-glancing}: points of $\Gamma_E$ are non-glancing, while points of $\mathcal G_E$ are glancing. The point needing proof is smoothness up to $\ol\Omega$ across the glancing set. Near a glancing point, the affine sections become tangent to $\p\Omega$, and in suitable local coordinates, the fiber domains shrink to a point.
	
	We record the elementary geometry of the projected sections. Since $\Omega$ is open and $P_E$ is an open linear map, the set $\Omega_E=P_E(\Omega)$ is open as a subset of $E$. Its boundary $\p\Omega_E$ is understood as the boundary relative to the vector space $E$. For $y\in\Omega_E$, the section $\Omega_{E,y}$ has a nonempty interior in the affine plane $y+V$. Points of $\p\Omega_E$ correspond to degenerate limiting sections.
	
	\begin{lemma}\label{lem:section-geometry}
		Let $\Omega\subset\R^n$ be a bounded smooth uniformly convex domain. Fix $E\in\Gr(k-1,n)$ and $V=E^\perp$.		
		\begin{enumerate}[\rm(i)]
			\item\label{item_1_sec_geo} If $y\in\Omega_E$, then every point of $\p\Omega\cap(y+V)$ is non-glancing, that is $|P_V\nu|>0$ there.
			\item\label{item_2_sec_geo} If $y\in\p\Omega_E$, then $\ol\Omega\cap(y+V)$ consists of at most one point. In particular, it has $q$-dimensional Hausdorff measure zero.
		\end{enumerate}
	\end{lemma}
	
	\begin{proof}
		We first prove \ref{item_1_sec_geo}. Let $y\in\Omega_E$. Since $y$ is in the projection of the open set $\Omega$, the section $\Omega\cap(y+V)$ contains an interior point of $\Omega$. 		Suppose that there is a point $x\in\p\Omega\cap(y+V)$ with $P_V\nu(x)=0$. Then $\nu(x)$ is orthogonal to $V$. Indeed, for every $v\in V$,
		\[
		v\cdot \nu(x)=v\cdot P_V\nu(x)=0.
		\]
		Since $T_x\p\Omega=\{\xi\in\R^n:\xi\cdot \nu(x)=0\}$, this gives $V\subset T_x\p\Omega$. Hence, the affine plane $x+V=y+V$ is contained in the tangent hyperplane to $\p\Omega$ at $x$. Uniform convexity implies strict convexity, and the tangent hyperplane supports $\ol\Omega$ at $x$ and meets $\ol\Omega$ only at $x$. This contradicts the fact that the section contains an interior point of $\Omega$. This proves $|P_V\nu|>0$ on $\p\Omega\cap(y+V)$.
		
		We next prove \ref{item_2_sec_geo}. Let $y\in\p\Omega_E$. Since $\Omega_E$ is an open convex subset of the Euclidean space $E$, the supporting hyperplane theorem gives a nonzero vector $\eta\in E$ such that $\eta\cdot (y'-y)\le 0$, for all $y'\in\Omega_E$. After normalizing $\eta$, we may assume $|\eta|=1$. Equivalently, $\eta\cdot y'\le \eta\cdot y$, for $y'\in\Omega_E$. For any $x'\in\Omega$, one has $P_Ex'\in\Omega_E$. Since $\eta\in E$ and $V=E^\perp$, there holds $\eta\cdot x'=\eta\cdot P_Ex'\le \eta\cdot y$. Hence, the hyperplane $\{x:\eta\cdot P_Ex=\eta\cdot y\}$ supports $\Omega$. The affine plane $y+V$ lies inside this supporting hyperplane. Uniform convexity implies that a supporting hyperplane touches $\ol\Omega$ at at most one point. Thus, $\ol\Omega\cap(y+V)$ consists of at most one point, and its $q$-dimensional Hausdorff measure is zero.
	\end{proof}
	
	\begin{lemma}\label{lem:b-sign}
		Let $\Omega\subset\R^n$ be a bounded smooth uniformly convex domain. Fix $E\in\Gr(k-1,n)$ and set $V=E^\perp$. Let $b_E$ be the section barrier defined in \eqref{eq:fiber-b}. For every nondegenerate section $\Omega_{E,y}$,
		\begin{equation}\label{eq:b-section-sign}
			-\frac{\diam(\Omega)^2}{2q}\le b_E(y,\cdot)\le0
			\quad
			\text{in } \Omega_{E,y}.
		\end{equation}
		The same bounds hold in the whole domain:
		\begin{equation}\label{eq:b-sign}
			-\frac{\diam(\Omega)^2}{2q}\le b_E\le0
			\quad
			\text{in } \Omega.
		\end{equation}
		If $x\in\p\Omega$ and $|P_V\nu(x)|>0$, then
		\begin{equation}\label{eq:b-Hopf}
			\p_\nu b_E(x)>0.
		\end{equation}
	\end{lemma}
	
	\begin{proof}
		We first work on a fixed section. Let $y\in\Omega_E$. The section $\Omega_{E,y}$ is a bounded smooth convex domain in the affine plane $y+V$. On this section, $b_E(y,\cdot)$ solves \eqref{eq:fiber-b}. The weak maximum principle in $\Omega_{E,y}$ gives $b_E(y,\cdot)\le0$ in $\Omega_{E,y}$. The strong maximum principle gives $b_E(y,\cdot)<0$ in the interior of every nondegenerate section.
		
		We next prove the lower bound. Set $D=\diam(\Omega)$ and choose a point $z_0\in\Omega_{E,y}$. Since $\Omega_{E,y}\subset\Omega$, every point $z\in\Omega_{E,y}$ satisfies $|z-z_0|\le D$. Define
		\begin{equation}\label{eq:section-quadratic-barrier}
			\beta(z)=\frac{|z-z_0|^2-D^2}{2q},
			\quad
			z\in y+V.
		\end{equation}
		Then $\Delta_V\beta=1$ in $y+V$. The function $b_E-\beta$ is harmonic in $\Omega_{E,y}$, and on $\p\Omega_{E,y}$ one has $b_E-\beta=-\beta\ge0$. The maximum principle gives $b_E\ge \beta$ in $\Omega_{E,y}$. Since $\beta\ge -D^2/(2q)$, we get $b_E(y,\cdot)\ge -\frac{D^2}{2q}$ in $\Omega_{E,y}$, which proves \eqref{eq:b-section-sign}. Every point of $\Omega$ lies in a nondegenerate section, so \eqref{eq:b-sign} follows.
		
		It remains to prove the boundary sign. Let $x\in\p\Omega$ satisfy $|P_V\nu(x)|>0$, and set $y=P_Ex$, then $x$ is a non-glancing boundary point of the section $\Omega_{E,y}$. Since $P_V\nu(x)\ne0$, the affine line $x+sP_V\nu(x)$ crosses $\p\Omega$ transversely at $x$. Hence, for one sign of small $s$, the points $x+sP_V\nu(x)$ lie in $\Omega$. These points have the same $E$-projection $y=P_Ex$, and $y\in\Omega_E$. The section boundary is smooth near $x$, and its outward unit normal inside the affine plane $y+V$ is $\nu_V:=\frac{P_V\nu}{|P_V\nu|}$. By the strong maximum principle, $b_E<0$ inside the section and $b_E=0$ on $\p\Omega_{E,y}$. The Hopf lemma in the section gives
		\begin{equation}\label{eq:Hopf-section}
			\p_{\nu_V}b_E(x)>0.
		\end{equation}
		
		We next justify the smoothness in the ambient variables near this non-glancing point. Let $x_0=x$ and set $y_0=P_Ex_0$. Since $x_0\in\p\Omega\cap(y_0+V)$ and $|P_V\nu(x_0)|>0$, the section crosses the boundary transversely at $x_0$. More explicitly, let $\rho$ be a smooth defining function for $\Omega$ near $\p\Omega$, with $\Omega=\{\rho>0\}$. In the splitting $x=y+z$, the boundary of the section is given by $\rho(y+z)=0$. The non-glancing condition is exactly
		\[
		D_z\rho(y_0+z_0)=P_V\nabla\rho(x_0)\neq0,
		\quad z_0=P_Vx_0.
		\]
		Thus, the implicit function theorem gives local boundary charts for $\p\Omega\cap(y+V)$ depending smoothly on $y$ near $y_0$.
		
		In fact, this can be done uniformly along the whole boundary of the section. By Lemma~\ref{lem:section-geometry}\ref{item_1_sec_geo}, every point of the compact set $\p\Omega\cap(y_0+V)$ is non-glancing. After shrinking a neighborhood $O\Subset\Omega_E$ of $y_0$, one has $|P_V\nu|\ge\lambda>0$ on $\p\Omega\cap(y+V)$ for all $y\in O$. A finite collection of the above implicit-function charts therefore shows that the domains $\Omega_{E,y}$ form a smooth family of bounded domains in the affine planes $y+V$, for $y\in O$. Equivalently, after choosing a reference section, there are smooth diffeomorphisms from this fixed section to $\Omega_{E,y}$, depending smoothly on $y$.
		
		Pulling back the sectionwise Dirichlet problem for $b_E$ to the fixed reference section gives a family of uniformly elliptic Dirichlet problems with smooth coefficients, smooth right-hand side, and zero boundary value. Standard parameter-dependent Schauder theory, or the implicit function theorem for the pulled-back Dirichlet operator, gives smooth dependence of $b_E(y,\cdot)$ on both the parameter $y$ and the fiber variable up to the section boundary. Thus $b_E$ is smooth in the ambient variables near $x_0$.
		
		Since $b_E=0$ on $\p\Omega$, its full gradient at $x$ is normal to $\p\Omega$: $\nabla b_E(x)=(\p_\nu b_E(x))\nu(x)$. Therefore,
		\begin{equation}\label{eq:normal-relation-b}
			\p_{\nu_V}b_E(x)=\nabla_Vb_E(x)\cdot\nu_V=|P_V\nu(x)|\p_\nu b_E(x).
		\end{equation}
		Since $|P_V\nu(x)|>0$, \eqref{eq:Hopf-section} gives $\p_\nu b_E(x)>0$.
	\end{proof}
	
	Smooth dependence of $w_E^f$ and $b_E$ on the section parameter is standard near $\Gamma_E$. The next lemmas deal with points of $\mathcal G_E$, where the sections shrink tangentially.

	\begin{remark}\label{rem:morse-lemma-parameters}
		We shall use the Morse lemma with parameters in the following elementary form. Let $F(y,z)$ be smooth near $(0,0)$, where $y$ is a parameter and $z$ is the variable. Suppose that $D_zF(0,0)=0$ and that $D_z^2F(0,0)$ is nondegenerate. Then, after shrinking neighborhoods, the critical point in the $z$ variable is a smooth function $z=c(y)$, and there is a smooth change of the $z$ variables, depending smoothly on $y$, which puts $F$ into a quadratic normal form in the $z$ variables. We use this with $F=\rho$ and with the negative definite quadratic form in the $z$ variables. See, for instance, \cite[Proposition 2.41]{Nicolaescu2011_Morse}.
	\end{remark}

	\begin{lemma}\label{lem:normal-form}
		Let $x_0\in\mathcal G_E$ be a glancing point and put $y_0=P_Ex_0$. There exist neighborhoods $O\subset E$ of $y_0$, $B\subset V$ of $0$, and $\mathcal U\subset\R^n$ of $x_0$, a smooth diffeomorphism $\Phi:O\times B\to\mathcal U$, and a smooth function $\beta:O\to\R$ with the following properties:
		\begin{enumerate}[\rm(i)]
			\item $\Phi(y,z)\in y+V$ for all $(y,z)\in O\times B$, and $\Phi(y_0,0)=x_0$.
			\item $\beta(y_0)=0$ and $d\beta(y_0)\ne0$.
			\item In these coordinates,
			\begin{equation}\label{eq:normal-form}
				\Phi^{-1}(\Omega\cap\mathcal U)=\{(y,z)\in O\times B: |z|^2<\beta(y)\}.
			\end{equation}
			\item If $\widetilde u=u\circ\Phi$, then the pullback of $\Delta_V$ has the form
			\begin{equation}\label{eq:Ly}
				(\Delta_Vu)\circ\Phi=L_y\widetilde u,
				\quad
				L_y=A^{\alpha\beta}(y,z)\p_{z_\alpha z_\beta}
				+B^\alpha(y,z)\p_{z_\alpha},
			\end{equation}
			where the coefficients are smooth, and $L_y$ is uniformly elliptic in the $z$ variables.
		\end{enumerate}
	\end{lemma}
	
	\begin{proof}
		We work in the splitting $\R^n=E\oplus V$ and write points as $(y,z)$, where $y\in E$ and $z\in V$. Translating the coordinates, we may assume that $x_0=0$ and $y_0=P_Ex_0=0$. Choose a smooth defining function in these coordinates. More explicitly, let $\rho=\rho(y,z)$ be a smooth real-valued function near $(0,0)$ such that, near $0$,
		\begin{equation}\label{eq:defining-function-local}
			\Omega=\{(y,z):\, \rho(y,z)>0\}, \quad
			\p\Omega=\{(y,z):\, \rho(y,z)=0\},
		\end{equation}
		and $D\rho\ne0$ on $\p\Omega$. Here $D_y\rho$ and $D_z\rho$ denote the derivatives of $\rho$ with respect to the base variables $y\in E$ and the fiber variables $z\in V$.
		
		Since $x_0$ is glancing, one has $P_V\nu(x_0)=0$, equivalently $V\subset T_{x_0}\p\Omega$. The differential $D\rho(0,0)$ is normal to $\p\Omega$ at $x_0$, and therefore it vanishes on every vector in $V$. In the variables $(y,z)$, this gives
		\begin{equation}\label{eq:Dzrho-zero}
			D_z\rho(0,0)=0.
		\end{equation}
		Since $\rho$ is a defining function, $D\rho(0,0)\ne0$. The $V$-component of this derivative is zero by \eqref{eq:Dzrho-zero}, so the $E$-component must be nonzero:
		\begin{equation}\label{eq:Dyrho-nonzero}
			D_y\rho(0,0)\ne0.
		\end{equation}
		
		We next check the nondegeneracy in the fiber variables. At a glancing point, the space $V$ is contained in $T_{x_0}\p\Omega$. With the sign convention $\Omega=\{\rho>0\}$, the second fundamental form of $\p\Omega$ is represented, up to a positive scalar factor, by $-D^2\rho(0,0)$ restricted to tangent directions. Uniform convexity therefore gives that $-D^2\rho(0,0)$ is positive definite on $T_{x_0}\p\Omega$. In particular,
		\begin{equation}\label{eq:Dzzrho-negative}
			-D_z^2\rho(0,0)>0
		\end{equation}
		as a quadratic form on $V$.
		
		By \eqref{eq:Dzzrho-negative}, the implicit function theorem applies to the equation $D_z\rho(y,z)=0$. Thus there is a unique smooth function $c(y)$ near $0$, with $c(0)=0$, such that $D_z\rho(y,c(y))=0$. Define
		\begin{equation}\label{eq:beta-def}
			\beta(y)=\rho(y,c(y)).
		\end{equation}
		Since $D_z\rho(y,c(y))=0$, differentiating \eqref{eq:beta-def} at $0$ gives $d\beta(0)=D_y\rho(0,0)$. By \eqref{eq:Dyrho-nonzero}, $d\beta(0)\ne0$.
		
		We apply the Morse lemma with parameters in the $z$ variables; see Remark \ref{rem:morse-lemma-parameters}. For each fixed $y$ close to $0$, the function $z\mapsto\rho(y,z)$ has a nondegenerate critical point at $z=c(y)$, and its Hessian in the $z$ variable is negative definite. The parameter-dependent Morse lemma gives a smooth change of fiber variables, depending smoothly on $y$, which moves this critical point to $z=0$ and writes the defining function as
		\begin{equation}\label{eq:morse}
			\rho(\Phi(y,z))=\beta(y)-|z|^2.
		\end{equation}
		The map is fiber-preserving, so $\Phi(y,z)\in y+V$ for each fixed $y$. Since $\Omega=\{\rho>0\}$ near $x_0$, \eqref{eq:morse} gives
		\begin{equation}\label{eq:normal-form-proof}
			\Phi^{-1}(\Omega\cap\mathcal U)=\{(y,z): |z|^2<\beta(y)\}
		\end{equation}
		after shrinking the neighborhoods.
		
		For fixed $y$, the map $z\mapsto\Phi(y,z)$ is a diffeomorphism between open subsets of the affine plane $y+V$. The operator $\Delta_V$ differentiates only along this affine plane. Its pullback by the fiber diffeomorphism is therefore a second-order operator only in the $z$ variables. By the chain rule, it has the form
		\begin{equation}\label{eq:Ly-proof}
			L_y=A^{\alpha\beta}(y,z)\p_{z_\alpha z_\beta}
			+B^\alpha(y,z)\p_{z_\alpha}.
		\end{equation}
		The coefficients are smooth because $\Phi$ is smooth. The principal matrix $(A^{\alpha\beta})$ is uniformly positive definite after shrinking the neighborhoods, since $D_z\Phi$ is invertible and the original operator is the Euclidean Laplacian on each fiber. This gives \eqref{eq:Ly}.
	\end{proof}

	In the next two local lemmas, the variables are $(y,t,z)$, where $y$ is a parameter, $t\ge0$ is a base variable, and $z$ is the fiber variable. For $t>0$, the fiber domain is the ball $|z|^2<t$. When $t=0$, this ball has radius zero and is reduced to the point $z=0$. We denote the corresponding zero-radius set by
	\begin{equation}\label{eq:zero-radius-set}
		\mathcal C=\{(y,t,z):\, t=0,\ z=0\}.
	\end{equation}
	A function is said to be flat on $\mathcal C$ if all its derivatives vanish to infinite order as $(t,z)\to(0,0)$, uniformly for $y$ in compact sets.

	\begin{lemma}\label{lem:flat-R}
		Let $y$ range in a small open set in $\R^m$, let $t\ge0$, and let $z\in\R^q$. Let $L_{y,t}$ be a smooth family of uniformly elliptic operators in the $z$ variables near $(y,t,z)=(0,0,0)$, of the form
		\begin{equation}\label{eq:Lyt}
			L_{y,t}=a^{ij}(y,t,z)\p_{z_i z_j}+a^i(y,t,z)\p_{z_i}.
		\end{equation}
		Assume that the matrix $(a^{ij})$ is uniformly positive definite near $(0,0,0)$. Let $R\in C^\infty$ be flat on the zero-radius set $\mathcal C$ in the following sense: for every multi-index $\alpha$ in the $y$ variables, every multi-index $\beta$ in the $z$ variables, every integer $a\ge0$, and every integer $N\ge0$,
		\begin{equation}\label{eq:R-flat}
			|\p_y^\alpha\p_t^a\p_z^\beta R(y,t,z)|\le C_{\alpha,a,\beta,N}(t+|z|^2)^N
		\end{equation}
		near $t=0$, $z=0$, for $t\ge0$. For $t>0$, let $W_f$ solve
		\begin{equation}\label{eq:flat-problem}
			\begin{cases}
				L_{y,t}W_f=R & \text{in } |z|^2<t,\\
				W_f=0 & \text{on } |z|^2=t.
			\end{cases}
		\end{equation}
		Then $W_f$ is flat on $\mathcal C$, together with all mixed derivatives. Here and below, the derivatives $\p_t^aW_f$ in the original variables are taken with $y$ and the original variable $z$ fixed. More explicitly, for every $\alpha,\beta$, every integer $a\ge0$, and every integer $N\ge0$,
		\begin{equation}\label{eq:Wf-flat}
			|\p_y^\alpha\p_t^a\p_z^\beta W_f(y,t,z)|
			\le C_{\alpha,a,\beta,N}t^N
			\quad
			\text{for } |z|\le\sqrt t
		\end{equation}
		when $t>0$ is sufficiently small.
	\end{lemma}
	
	\begin{proof}
		We rescale the shrinking ball to a fixed ball. Put $z=\sqrt t\,\zeta$ and define
		\begin{equation}\label{eq:scaled-W}
			\widetilde W_f(y,t,\zeta)=W_f(y,t,\sqrt t\,\zeta),
			\quad
			|\zeta|<1.
		\end{equation}
		Since $\p_{z_i}=t^{-1/2}\p_{\zeta_i}$ and $\p_{z_i z_j}=t^{-1}\p_{\zeta_i\zeta_j}$, multiplying the equation by $t$ gives
		\begin{equation}\label{eq:scaled-flat-equation}
			\begin{cases}
				\widetilde L_{y,t}\widetilde W_f=tR(y,t,\sqrt t\,\zeta) & \text{ in }|\zeta|<1, \\
				\widetilde W_f=0 & \text{ on } |\zeta|=1,
			\end{cases}
		\end{equation}
		where
		\[
		\widetilde L_{y,t}
		=
		a^{ij}(y,t,\sqrt t\,\zeta)\p_{\zeta_i\zeta_j}
		+
		t^{1/2}a^i(y,t,\sqrt t\,\zeta)\p_{\zeta_i}.
		\]
		For $0<t<t_0$, after shrinking the neighborhood if necessary, the operators $\widetilde L_{y,t}$ are uniformly elliptic on the fixed ball $B_1$. Their coefficients are uniformly bounded in every $C^m(\ol{B_1})$ norm when no $t$ derivatives are taken. When $t$ derivatives are taken, the only possible singular factors come from differentiating $\sqrt t\,\zeta$. Thus, any fixed finite number of $t$ derivatives of the coefficients is bounded by a fixed finite power of $t^{-1}$.
		
		We first justify differentiating the rescaled solution with respect to the parameters. Fix $0<t_1<t_2$ and a compact set of $y$ parameters. On this parameter set, the Dirichlet realizations
		\begin{equation*}
			\widetilde L_{y,t}:\{U\in C^{m+2,\alpha_0}(\ol{B_1}):U|_{\p B_1}=0\}\to C^{m,\alpha_0}(\ol{B_1})
		\end{equation*}
		are isomorphisms, by the maximum principle and Schauder theory, and their coefficients depend smoothly on $(y,t)$ for $t>0$. Therefore, the solution $\widetilde W_f$ depends smoothly on $(y,t)$ for every $t>0$. Equivalently, one may obtain the same differentiated equations by applying difference quotients in the parameters and passing to the limit by Schauder estimates. In the argument below, all parameter differentiations are first made for $t>0$, and the estimates obtained are uniform as $t\downarrow0$.
		
		We next check the rescaled right-hand side. Fix the number of derivatives in $y,t,\zeta$ and fix an integer $M\ge0$. Derivatives in $\zeta$ give factors of $\sqrt t$ and derivatives in $z$. Derivatives in $t$ produce finitely many factors of $t^{-1/2}$ from the composition $z=\sqrt t\,\zeta$. Since $|\zeta|<1$, one has $t+|\sqrt t\,\zeta|^2\le 2t$. The flatness assumption \eqref{eq:R-flat} is available with arbitrarily large order. Choosing an order large enough absorbs all finite powers of $t^{-1/2}$. Therefore, for every integer $m\ge0$ and every $M\ge0$,
		\begin{equation}\label{eq:scaled-R-flat}
			\big\|\p_y^\alpha\p_t^a\p_\zeta^\beta(tR(y,t,\sqrt t\,\zeta))\big\|_{C^{m,\alpha_0}(\ol{B_1})}
			\le C_{\alpha,a,\beta,m,M}t^M,
		\end{equation}
		where $\alpha_0\in(0,1)$ is fixed.
		
		Apply the Schauder estimate for the Dirichlet problem on the fixed ball $B_1$ to \eqref{eq:scaled-flat-equation}. Since the boundary value is zero and the undifferentiated coefficients are uniformly elliptic with uniformly controlled $C^{m,\alpha_0}$ norms in the $\zeta$ variables, \eqref{eq:scaled-R-flat} gives
		\begin{equation}\label{eq:scaled-Schauder-0}
			\big\|\widetilde W_f\big\|_{C^{m+2,\alpha_0}(\ol{B_1})}\le C_{m,M}t^M
		\end{equation}
		for every $m$ and every $M$.
		
		Let $U_{\alpha,a}=\p_y^\alpha\p_t^a\widetilde W_f$. Since $\widetilde W_f=0$ on $|\zeta|=1$ for every $y$ and $t$, each $U_{\alpha,a}$ also has zero boundary value on $|\zeta|=1$. Differentiating \eqref{eq:scaled-flat-equation} gives
		\begin{equation}\label{eq:differentiated-scaled-equation}
			\begin{cases}
				\widetilde L_{y,t}U_{\alpha,a}
				=
				\p_y^\alpha\p_t^a(tR(y,t,\sqrt t\,\zeta))+\mathcal E_{\alpha,a}
				&\text{in }|\zeta|<1,\\
				U_{\alpha,a}=0&\text{on }|\zeta|=1.
			\end{cases}
		\end{equation}
		Here $\mathcal E_{\alpha,a}$ is a finite sum of terms of the form
		\[
		C(y,t,\zeta)\,\p_\zeta^\gamma\p_y^{\alpha'}\p_t^{a'}\widetilde W_f,
		\quad
		|\alpha'|+a'<|\alpha|+a,
		\]
		where the coefficients $C(y,t,\zeta)$ are derivatives of the coefficients of $\widetilde L_{y,t}$. For any fixed differentiated equation, these coefficients are bounded by some finite power of $t^{-1}$ in the required $C^{m,\alpha_0}(B_1)$ norms.
		
		We prove by induction on $|\alpha|+a$ that
		\begin{equation}\label{eq:scaled-flat-all}
			\|\p_y^\alpha\p_t^a\widetilde W_f\|_{C^{m,\alpha_0}(B_1)}
			\le C_{\alpha,a,m,M}t^M
		\end{equation}
		for every $m$ and every $M$. The case $|\alpha|+a=0$ is \eqref{eq:scaled-Schauder-0}. For the induction step, the first term on the right-hand side of \eqref{eq:differentiated-scaled-equation} is $O(t^M)$ by \eqref{eq:scaled-R-flat}, after increasing the flatness order of $R$. The error term $\mathcal E_{\alpha,a}$ is a finite sum of lower parameter derivatives of $\widetilde W_f$, multiplied by coefficients with only finite powers of $t^{-1}$. By the induction hypothesis, the lower derivatives are $O(t^{M'})$ for every $M'$. Taking $M'$ sufficiently large absorbs the finite losses from the differentiated coefficients. Schauder estimates for \eqref{eq:differentiated-scaled-equation} give \eqref{eq:scaled-flat-all}.
		
		It remains to return to the original variables. Let $D_t^z$ denote differentiation with $y$ and the original variable $z$ fixed. Let $D_t^\zeta$ denote differentiation with $y$ and the rescaled variable $\zeta$ fixed. Since $\zeta=z/\sqrt t$, when $z$ is fixed one has $D_t^z\zeta=-(2t)^{-1}\zeta$. By the chain rule,
		\begin{equation}\label{eq:t-fixed-z}
			D_t^z W_f(y,t,z)=
			\Big(D_t^\zeta\widetilde W_f-\frac{1}{2t}\zeta\cdot\nabla_\zeta\widetilde W_f\Big)
			\big(y,t,\frac{z}{\sqrt t}\big).
		\end{equation}
		Also,
		\begin{equation}\label{eq:z-derivative-loss}
			\p_z^\beta W_f(y,t,z)
			=
			t^{-|\beta|/2}\p_\zeta^\beta\widetilde W_f
			\big(y,t,\frac{z}{\sqrt t}\big).
		\end{equation}
		Iterating \eqref{eq:t-fixed-z} and \eqref{eq:z-derivative-loss}, every mixed derivative $\p_y^\alpha(D_t^z)^a\p_z^\beta W_f(y,t,z)$ is a finite sum of terms of the form
		\begin{equation}\label{eq:back-transform-terms}
			t^{-N_0}C(\zeta)\p_y^{\alpha'}(D_t^\zeta)^{a'}\p_\zeta^{\beta'}\widetilde W_f(y,t,\zeta),
			\quad
			\zeta=\frac{z}{\sqrt t},
			\quad
			|\zeta|<1.
		\end{equation}
		Here $N_0$ depends only on $\alpha,a,\beta$, and $C(\zeta)$ is smooth for $|\zeta|<1$. Since \eqref{eq:scaled-flat-all} holds with an arbitrary power $t^M$, taking $M$ larger before applying \eqref{eq:back-transform-terms} absorbs all finite losses. This gives \eqref{eq:Wf-flat}. The estimate also shows that the extension obtained by setting $W_f=0$ on $\mathcal C$ has all mixed derivatives equal to zero there. Thus, the zero extension is smooth across $\mathcal C$. Since the estimates above are obtained in boundary H\"older norms on $\ol{B_1}$, the same flatness estimates hold up to the lateral boundary $|z|^2=t$ for $t>0$.
	\end{proof}

	\begin{remark}\label{rem:Borel-theorem}
		We use the following standard form of Borel's theorem. If one prescribes an arbitrary formal Taylor series in variables $(t,z)$ whose coefficients are smooth functions of an additional parameter $y'$, then there exists a smooth function of $(y',t,z)$ realizing this formal Taylor series. In particular, if two smooth functions have the same prescribed Taylor series in $(t,z)$ at $(t,z)=(0,0)$ for every $y'$, then their difference is flat on the set $\{(y',t,z):t=0,\ z=0\}$.
	\end{remark}
	
	\begin{lemma}\label{lem:shrinking-section}
		Let $U\subset\R^m$ be open, let $0\in U$, and let $\mathcal D=\{(y,z):|z|^2<\beta(y)\}$, where $\beta\in C^\infty(U)$, $\beta(0)=0$, and $d\beta(0)\ne0$. Let $L_y$ be a smooth family of uniformly elliptic operators in the $z$ variables,
		\begin{equation}\label{eq:shrinking-Ly}
			L_y=a^{ij}(y,z)\p_{z_i z_j}+a^i(y,z)\p_{z_i}.
		\end{equation}
		Let $F$ be smooth near $(0,0)$. For $\beta(y)>0$, let $W(y,\cdot)$ solve
		\begin{equation}\label{eq:shrinking-problem}
			\begin{cases}
				L_yW(y,\cdot)=F(y,\cdot) & \text{ in } |z|^2<\beta(y), \\
				W(y,\cdot)=0 &\text{ on } |z|^2=\beta(y).
			\end{cases}
		\end{equation}
		Then $W$ is $C^\infty$ up to the boundary near $(0,0)$, in the sense that every derivative in the variables $(y,z)$ has a continuous extension to $\ol{\mathcal D}$ near $(0,0)$.
	\end{lemma}
	
	\begin{proof}
		Since $d\beta(0)\ne0$, we use $t=\beta(y)$ as one of the base coordinates and write the remaining base variables as $y'$. Thus, the collapsing sections are $|z|^2<t$, $t\ge0$. Set $r=|z|^2$. Since second derivatives commute, only the symmetric part of the principal coefficients enters the operator. Thus, replacing $a^{ij}$ by $(a^{ij}+a^{ji})/2$, we may assume that $a^{ij}=a^{ji}$.
		
		We first construct a smooth approximate solution whose boundary value vanishes on $r=t$. We look for it in the form $W_0=(r-t)\Psi$. Define $\mathcal T\Psi:=L_{y',t}\big((r-t)\Psi\big)$, then a direct computation gives
		\begin{equation}\label{eq:T-expanded}
			\mathcal T\Psi=(r-t)L_{y',t}\Psi+2a^{ij}(y',t,z)(\p_{z_i}r)\p_{z_j}\Psi+(L_{y',t}r)\Psi.
		\end{equation}
		At $t=0$, $z=0$, one has
		\begin{equation}\label{eq:Lr-positive}
			\left. L_{y',0}r\right|_{z=0}=2\tr(a^{ij}(y',0,0))>0.
		\end{equation}
		
		We now construct the Taylor series of $\Psi$ in the variables $(t,z)$, with coefficients depending smoothly on $y'$. For each integer $N\ge0$, let $\mathcal H_N$ be the finite-dimensional space spanned by the monomials $t^a z^\gamma$, $2a+|\gamma|=N$. Let $\Pi_N$ denote the projection of a formal Taylor series in $(t,z)$ onto the sum of these monomials.
		
		Set $A^{ij}(y'):=a^{ij}(y',0,0)$, $L_{0,y'}:=A^{ij}(y')\p_{z_i z_j}$, and $\mathcal T_{0,y'}\Psi:=L_{0,y'}\big((r-t)\Psi\big)$. We first record the following elementary fact. If $P\in\mathcal H_N$, then
		\begin{equation}\label{eq:PiN-T}
			\Pi_N(\mathcal TP)=\mathcal T_{0,y'}P.
		\end{equation}
		Indeed, the part of $\mathcal T$ in which the second-order coefficients are frozen at $(t,z)=(0,0)$ and only the second-order operator is kept is exactly $\mathcal T_{0,y'}$. All other terms have no contribution to $\Pi_N(\mathcal TP)$. More explicitly, terms coming from $a^{ij}(y',t,z)-A^{ij}(y')$ contain at least one factor of either $z$ or $t$, and hence contribute only to monomials with $2a+|\gamma|>N$. The first-order part of $L_{y',t}$ also contributes only to monomials with $2a+|\gamma|>N$: in $(r-t)a^i\p_{z_i}P$ the factor $(r-t)$ contributes two powers relative to the relation $r=t$, while $\p_{z_i}$ lowers the $z$ degree by one; in $a^i(\p_{z_i}r)P$ the factor $\p_{z_i}r=2z_i$ raises the $z$ degree by one. This proves \eqref{eq:PiN-T}.
		
		We next show that $\mathcal T_{0,y'}:\mathcal H_N\to\mathcal H_N$ is an isomorphism for every $N$. Suppose that $P\in\mathcal H_N$ and $\mathcal T_{0,y'}P=0$. For each fixed $t>0$, set $Z(t,z)=(r-t)P(t,z)$, then
		\[
		\begin{cases}
			L_{0,y'}Z(t,\cdot)=0 & \text{ in } |z|^2<t,\\
			Z(t,\cdot)=0 & \text{ on } |z|^2=t.
		\end{cases}
		\]
		The operator $L_{0,y'}$ is uniformly elliptic with constant coefficients in the $z$ variables. By the maximum principle in the ball $|z|^2<t$, one has $Z(t,\cdot)=0$ for every $t>0$. Since $r-t$ is not identically zero in the open set $\{t>0,\ |z|^2<t\}$, this implies that $P=0$ there, and then $P$ is the zero polynomial. Thus, $\mathcal T_{0,y'}$ is injective on the finite-dimensional space $\mathcal H_N$, and it is an isomorphism.
		
		We now determine the Taylor coefficients of $\Psi$ recursively. Write the formal Taylor expansion of $F$ in the form $F\sim \sum_{N=0}^\infty F_N$ with $F_N\in\mathcal H_N$, where the coefficients of each $F_N$ are smooth functions of $y'$. Suppose that $\Psi_0,\ldots,\Psi_{N-1}$ have already been chosen, with $\Psi_\ell\in\mathcal H_\ell$, so that
		\[
		\Pi_\ell\Big(\mathcal T\sum_{j=0}^{N-1}\Psi_j-F\Big)=0 \quad\text{for }0\le \ell\le N-1.
		\]
		At level $N$, the only unknown contribution is $\mathcal T_{0,y'}\Psi_N$, by \eqref{eq:PiN-T}. All other terms of level $N$ are already determined by $\Psi_0,\ldots,\Psi_{N-1}$ and by the Taylor coefficients of the operator. Therefore, we choose $\Psi_N\in\mathcal H_N$ as the unique solution of
		\[
		\mathcal T_{0,y'}\Psi_N =F_N-\Pi_N\Big(\mathcal T\sum_{j=0}^{N-1}\Psi_j\Big).
		\]
		This determines $\Psi_N$ uniquely. In a fixed monomial basis of $\mathcal H_N$, the matrix of $\mathcal T_{0,y'}$ depends smoothly on $y'$, and it is invertible. Hence, its inverse also depends smoothly on $y'$, and the coefficient $\Psi_N$ depends smoothly on $y'$.
		
		By induction in $N$, we obtain a formal Taylor series $\sum_{N=0}^\infty \Psi_N$ in the variables $(t,z)$, with coefficients smooth in $y'$, such that the formal Taylor series of $\mathcal T\Psi$ agrees with that of $F$. By Borel's theorem with parameters, there exists a smooth function $\Psi_{\rm app}(y',t,z)$ whose Taylor series in $(t,z)$ at $(0,0)$ is this formal series. This implies that $F-\mathcal T\Psi_{\rm app}$ is flat at $t=0$, $z=0$, with all derivatives in $(y',t,z)$.
		
		Set $W_0=(r-t)\Psi_{\rm app}$, then $W_0=0$ on $r=t$, and
		\[
		L_{y',t}W_0=\mathcal T\Psi_{\rm app}=F-R,
		\]
		where $R:=F-\mathcal T\Psi_{\rm app}$ satisfies the flatness condition in Lemma~\ref{lem:flat-R} at the collapsed set $\{t=0,\ z=0\}$.
		
		Let $W_f=W-W_0$, then $W_f$ solves
		\[
		\begin{cases}
			L_{y',t}W_f=R & \text{in } |z|^2<t,\\
			W_f=0 & \text{on } |z|^2=t.
		\end{cases}
		\]
		By Lemma \ref{lem:flat-R}, $W_f$ is flat at the zero-radius set, together with all mixed derivatives, up to the lateral boundary $|z|^2=t$. Hence, the zero extension of $W_f$ is smooth through $t=0$, $z=0$. Since $W_0$ is smooth, $W=W_0+W_f$ is smooth up to the collapsed section.
		
		Away from the zero-radius set, the boundary $|z|^2=t$ is an ordinary smooth moving boundary. Standard parameter-dependent elliptic boundary regularity gives smoothness there. Combining this with the collapsed-section analysis proves the claimed smooth extension near $(0,0)$.
	\end{proof}
	
	\begin{proposition}\label{prop:smooth-corrections}
				Let $\Omega\subset\R^n$ be a bounded smooth uniformly convex domain. Fix $E\in\Gr(k-1,n)$ and set $V=E^\perp$. If $f\in C^\infty(\ol\Omega)$, then the functions $w_E^f$ and $b_E$ defined section by section in \eqref{eq:fiber-w} and \eqref{eq:fiber-b} extend to functions in $C^\infty(\ol\Omega)$. Moreover,
		\begin{equation}\label{eq:w_E^f}
			\begin{cases}
				\Delta_Vw_E^f=f & \text{ in }\Omega, \\
				w_E^f =0 &\text{ on }\p \Omega,
			\end{cases}
		\end{equation}
		and 
		\begin{equation}\label{eq:b_E}
			\begin{cases}
				\Delta_Vb_E=1 & \text{ in }\Omega, \\
				b_E=0 &\text{ on }\p\Omega.
			\end{cases}
		\end{equation}
	\end{proposition}
	
	\begin{proof}
		We first treat all points whose section parameter lies in $\Omega_E$. Fix $y_0\in\Omega_E$. By Lemma~\ref{lem:section-geometry}\ref{item_1_sec_geo}, every point of $\p\Omega\cap(y_0+V)$ is non-glancing. Since $\p\Omega\cap(y_0+V)$ is compact, there are a neighborhood $O_0\Subset\Omega_E$ of $y_0$ and a constant $\lambda>0$ such that $|P_V\nu|\ge \lambda$ on $\p\Omega\cap(y+V)$ for every $y\in O_0$. The implicit function theorem applied to a defining function of $\p\Omega$ then gives boundary charts for the domains $\Omega_{E,y}$ depending smoothly on $y\in O_0$. Equivalently, the sections $\Omega_{E,y}$ form a smooth family of bounded domains in the affine planes $y+V$.
		
		In these smooth families of domains, the Dirichlet problems for $\Delta_V$ have smooth coefficients, smooth right-hand sides, and zero boundary values. Standard parameter-dependent elliptic regularity gives that $w_E^f$ and $b_E$ depend smoothly on $y$ and on the fiber variable, up to the section boundary. This proves smoothness near all points lying over $O_0$. Since $y_0\in\Omega_E$ was arbitrary, this covers the non-glancing part and all interior points of $\Omega$.
		
		It remains to treat glancing points. Let $x_0\in\mathcal G_E$ and set $y_0=P_Ex_0$. Since $x_0\in\p\Omega$ and $P_Ex_0=y_0$, one has $y_0\in \ol{\Omega_E}$. If $y_0\in\Omega_E$, then Lemma \ref{lem:section-geometry}\ref{item_1_sec_geo} would imply that every point of $\p\Omega\cap(y_0+V)$ is non-glancing. This contradicts $x_0\in\mathcal G_E$. Hence $y_0\in\p\Omega_E$.
		
		By Lemma~\ref{lem:section-geometry}\ref{item_2_sec_geo}, the set $\ol\Omega\cap(y_0+V)$ consists of at most one point. Since it contains $x_0$, it is exactly $\{x_0\}$. Apply Lemma~\ref{lem:normal-form} at $x_0$, and let $\mathcal U$ be the corresponding normal-form neighborhood. There is a neighborhood $O$ of $y_0$ in $E$ such that
		\begin{equation}\label{eq:nearby-sections-in-chart}
			\ol\Omega\cap(y+V)\subset \mathcal U
			\quad
			\text{for all } y\in O\cap\ol\Omega_E.
		\end{equation}
		Indeed, otherwise one could find $y_\ell\to y_0$ and $x_\ell\in\ol\Omega\cap(y_\ell+V)$ with $x_\ell\notin \mathcal U$. Passing to a subsequence gives a point in $\ol\Omega\cap(y_0+V)=\{x_0\}$ outside $\mathcal U$, a contradiction.
		
		In the fiber-preserving coordinates given by Lemma~\ref{lem:normal-form}, the nearby full sections are $|z|^2<\beta(y)$, and the operator $\Delta_V$ pulls back to a smooth uniformly elliptic family $L_y$ in the $z$ variables. Lemma \ref{lem:shrinking-section} applies to the transformed source $f\circ\Phi$ and gives smoothness of $w_E^f$ near $x_0$. Applying the same lemma with the transformed right-hand side $1$ gives smoothness of $b_E$. The local solutions agree with the original sectionwise Dirichlet solutions on all nearby nondegenerate sections, since \eqref{eq:nearby-sections-in-chart} ensures that the local section is the full section. The local extensions, therefore, patch with the non-glancing extensions. The equations and boundary conditions follow by continuity from the corresponding sectionwise Dirichlet problems on nondegenerate sections.
	\end{proof}
	
	The flux identity below is the bridge from the boundary asymptotics to the Radon transform.
	
	\begin{lemma}\label{lem:flux}
		Let $y\in\Omega_E$ be such that $\Omega_{E,y}$ is nondegenerate, then
		\begin{equation}\label{eq:flux}
			\int_{\Omega\cap(y+V)} f\,d\Hq
			=\int_{\p\Omega\cap(y+V)} |P_V\nu(x)|\p_\nu w_E^f(x)\,dS_{q-1}(x).
		\end{equation}
	\end{lemma}
	
	\begin{proof}
		By the fiberwise equation and the divergence theorem in the affine plane $y+V$,
		\begin{equation}\label{eq:flux-proof-1}
			\int_{\Omega_{E,y}} f\,d\Hq
			=\int_{\p\Omega_{E,y}}\p_{\nu_V}w_E^f\,dS_{q-1}.
		\end{equation}
		Since $w_E^f=0$ on $\p\Omega$, its full gradient is normal to $\p\Omega$, so $\nabla w_E^f=(\p_\nu w_E^f)\nu$ on $\p\Omega$. Using $\nu_V=P_V\nu/|P_V\nu|$, we get $\p_{\nu_V}w_E^f=|P_V\nu|\p_\nu w_E^f$. Substitution gives \eqref{eq:flux}.
	\end{proof}

	\section{Large rank \texorpdfstring{$k-1$}{k-1} asymptotics}\label{sec:asymp}
	
	We prove the asymptotics for the solution with boundary value $t\phi_E$. The algebraic input is the expansion of $\sigma_k$ near the rank $k-1$ matrix $P_E$.
	
	\begin{lemma}\label{lem:sigma-expansion}
		Let $H$ range over a bounded subset of $\Sym(n)$. As $t\to\infty$,
		\begin{equation}\label{eq:sigma-j-expansion}
			\sigma_j(tP_E+t^{1-k}H)=\binom{k-1}{j}t^j+O(t^{j-k})
			\quad \text{for } 1\le j\le k-1,
		\end{equation}
		and
		\begin{equation}\label{eq:sigma-k-expansion}
			\sigma_k(tP_E+t^{1-k}H)=\tr(P_VH)+t^{-k}Q_E(H)+O(t^{-2k}).
		\end{equation}
		Here, $Q_E$ is a homogeneous polynomial of degree two in the entries of $H$, and the remainders are uniform when $H$ ranges over bounded sets.
	\end{lemma}
	
	\begin{proof}
		The statement is invariant under orthogonal changes of coordinates. We may assume that
		\begin{equation}\label{eq:E-standard-coordinates}
			E=\operatorname{span}\{e_1,\ldots,e_{k-1}\},
			\quad
			V=\operatorname{span}\{e_k,\ldots,e_n\}.
		\end{equation}
		In these coordinates,
		\begin{equation}\label{eq:PE-standard}
			P_E=\operatorname{diag}(\underbrace{1,\ldots,1}_{k-1},0,\ldots,0).
		\end{equation}
		Set $s=t^{-k}$, by homogeneity, then 
		\begin{equation}\label{eq:homogeneous-reduction}
			\sigma_j(tP_E+t^{1-k}H)=t^j\sigma_j(P_E+sH).
		\end{equation}
		
		We first consider $1\le j\le k-1$. Since $P_E$ has exactly $k-1$ eigenvalues equal to $1$ and the remaining eigenvalues equal to $0$, $\sigma_j(P_E)=\binom{k-1}{j}$. Since $\sigma_j$ is a polynomial in the entries of the matrix, and $H$ ranges in a bounded set,
		\begin{equation}\label{eq:sigma-j-taylor}
			\sigma_j(P_E+sH)=\binom{k-1}{j}+O(s).
		\end{equation}
		Multiplying by $t^j$ and using $s=t^{-k}$ gives \eqref{eq:sigma-j-expansion}.
		
		It remains to compute the first nonzero term in $\sigma_k(P_E+sH)$. We use the principal minor formula
		\begin{equation}\label{eq:principal-minor-formula}
			\sigma_k(A)=\sum_{|I|=k}\det A_{I,I},
		\end{equation}
		where $A_{I,I}$ denotes the principal submatrix with row and column set $I$. The only principal minors that can contribute to the coefficient of $s$ are those whose index set contains all $k-1$ directions in $E$ and one direction in $V$. These index sets are
		\begin{equation}\label{eq:leading-minors}
			I_\alpha=\{1,\ldots,k-1,\alpha\},
			\quad
			\alpha=k,\ldots,n.
		\end{equation}
		For such an $I_\alpha$, the first $k-1$ indices belong to $E$, and the last index $\alpha$ belongs to $V$. Since $P_E$ is the orthogonal projection onto $E$, it is the identity on the first $k-1$ directions, and it is zero on the $V$ direction $e_\alpha$. Hence,
		\begin{equation}\label{eq:minor-PE}
			(P_E)_{I_\alpha,I_\alpha}
			=
			\operatorname{diag}(1,\ldots,1,0).
		\end{equation}
		The final zero is the entry corresponding to the $V$ direction $e_\alpha$.
		
		For this minor,
		\begin{equation}\label{eq:minor-expansion}
			\det\big((P_E+sH)_{I_\alpha,I_\alpha}\big)=sH_{\alpha\alpha}+O(s^2).
		\end{equation}
		Indeed, at $s=0$ the matrix has diagonal entries $1,\ldots,1,0$. The only cofactor that is nonzero is the cofactor of the zero diagonal entry, and that cofactor is the product of the $k-1$ ones. Thus, the linear term is exactly $H_{\alpha\alpha}$.
		
		If a $k$-element set $I$ does not contain all indices $1,\ldots,k-1$, then $(P_E)_{I,I}$ has rank at most $k-2$. The first derivative of the determinant of such a matrix is zero because all cofactors vanish. These minors start only at order $s^2$.
		
		Adding the linear contributions from \eqref{eq:leading-minors}, the coefficient of $s$ in $\sigma_k(P_E+sH)$ is $\sum_{\alpha=k}^n H_{\alpha\alpha}=\tr(P_VH)$. If $D\sigma_k(P_E)[H]$ denotes the directional derivative
		\begin{equation}\label{eq:D-sigma-definition}
			D\sigma_k(P_E)[H]
			=
			\frac{d}{ds}\Big|_{s=0}\sigma_k(P_E+sH),
		\end{equation}
		then
		\begin{equation}\label{eq:first-var-sigmak-identity}
			D\sigma_k(P_E)[H]=\tr(P_VH).
		\end{equation}
		
		Since $\sigma_k(P_E)=0$, the Taylor expansion of $\sigma_k(P_E+sH)$ begins with the linear term just computed. The coefficient of $s^2$ is a homogeneous quadratic polynomial in $H$, which we denote by $Q_E(H)$. The remaining terms are $O(s^3)$ uniformly for $H$ in bounded sets. Thus,
		\begin{equation}\label{eq:sigma-k-taylor}
			\sigma_k(P_E+sH)
			=
			s\tr(P_VH)+s^2Q_E(H)+O(s^3).
		\end{equation}
		Multiplying by $t^k$ and using $s=t^{-k}$ gives
		\begin{equation}\label{eq:sigma-k-final}
			\sigma_k(tP_E+t^{1-k}H)
			=
			\tr(P_VH)+t^{-k}Q_E(H)+O(t^{-2k}).
		\end{equation}
		This proves \eqref{eq:sigma-k-expansion}.
	\end{proof}
	
	For the rest of this section, fix $E\in\Gr(k-1,n)$ and a positive source $f\in C^\infty(\ol\Omega)$. Recall that $\phi_E(x)=|P_Ex|^2/2$, so $D^2\phi_E=P_E$. By Proposition \ref{prop:smooth-corrections}, the fiberwise Poisson correction $w_E^f$ belongs to $C^\infty(\ol\Omega)$, satisfies $w_E^f=0$ on $\p\Omega$, and satisfies $\Delta_Vw_E^f=\tr(P_VD^2w_E^f)=f$ in $\Omega$.
	
	We define
	\begin{equation}\label{eq:U-def}
		U_{t,E}^f=t\phi_E+t^{1-k}w_E^f.
	\end{equation}
	This function has boundary value $t\phi_E$ on $\p\Omega$, and $D^2U_{t,E}^f=tP_E+t^{1-k}D^2w_E^f$. Applying Lemma \ref{lem:sigma-expansion} with $H=D^2w_E^f$ gives $\sigma_k(D^2U_{t,E}^f)=f+O(t^{-k})$. 
	
	\begin{lemma}\label{lem:admissibility}
		Let $W\in C^2(\ol\Omega)$ satisfy $\Delta_VW\ge c_1>0$, then $D^2(t\phi_E+t^{1-k}W)\in\Gk$, for all sufficiently large $t$.
	\end{lemma}
	
	\begin{proof}
		For $1\le j\le k-1$, \eqref{eq:sigma-j-expansion} gives $\sigma_j(D^2(t\phi_E+t^{1-k}W))>0$ for large $t$. For $j=k$, \eqref{eq:sigma-k-expansion} gives
		\begin{equation}\label{eq:admissible-sigmak}
			\sigma_k(D^2(t\phi_E+t^{1-k}W))=\Delta_VW+O(t^{-k})>0
		\end{equation}
		for large $t$. Hence, the Hessian lies in $\Gk$.
	\end{proof}
	
	\begin{proposition}\label{prop:global-C0}
		Fix $E\in\Gr(k-1,n)$ and let $f\in C^\infty(\ol\Omega)$ satisfy $f\ge c_0>0$. For each $t>0$, let $u_{t,E}^f\in C^\infty(\ol\Omega)$ be the unique smooth $k$-admissible solution of
		\begin{equation}\label{eq:u-t-E}
			\begin{cases}
				\sigma_k(D^2u_{t,E}^f)=f & \text{in } \Omega,\\
				u_{t,E}^f=t\phi_E & \text{on } \p\Omega.
			\end{cases}
		\end{equation}
		Let $U_{t,E}^f$ be defined by \eqref{eq:U-def}. Then $U_{t,E}^f\in C^\infty(\ol\Omega)$ and $U_{t,E}^f=t\phi_E$ on $\p\Omega$. In particular, $u_{t,E}^f-U_{t,E}^f\in C^\infty(\ol\Omega)$. There are constants $C>0$ and $t_0>0$, depending on $E$, $f$, $\Omega$, and $k$, but not on $t$, such that
		\begin{equation}\label{eq:C0-asymp}
			\big\|u_{t,E}^f-U_{t,E}^f\big\|_{L^\infty(\Omega)}
			\le Ct^{1-2k} \quad \text{for } t\ge t_0.
		\end{equation}
	\end{proposition}
	
	\begin{proof}
		The regularity of $u_{t,E}^f$ follows from Proposition \ref{prop:CNS-forward}. The function $w_E^f$ is smooth up to $\p\Omega$ by Proposition \ref{prop:smooth-corrections}, and $w_E^f=0$ on $\p\Omega$. Therefore, $U_{t,E}^f\in C^\infty(\ol\Omega)$ and $U_{t,E}^f=t\phi_E$ on $\p\Omega$.
		
		For $\sigma\in\{-1,1\}$, define
		\begin{equation}\label{eq:global-barriers}
			V_t^\sigma=t\phi_E+t^{1-k}\big(w_E^f+\sigma C_*t^{-k}b_E\big),
		\end{equation}
		where $C_*>0$ will be chosen later. Since $b_E\in C^\infty(\ol\Omega)$ satisfies \eqref{eq:b_E}, the barriers $V_t^\sigma$ are smooth up to the boundary and satisfy $V_t^\sigma=t\phi_E$ on $\p\Omega$.
		
		Applying Lemma \ref{lem:sigma-expansion} with $H=D^2w_E^f+\sigma C_*t^{-k}D^2b_E$ gives 
		\begin{equation}\label{eq:barrier-source-expansion}
			\sigma_k(D^2V_t^\sigma)=f+\sigma C_*t^{-k}+t^{-k}Q_E\big(D^2w_E^f+\sigma C_*t^{-k}D^2b_E\big)
			+O_{C_*}(t^{-2k})
		\end{equation}
		uniformly in $\Omega$. Since $Q_E$ is quadratic and $D^2w_E^f,D^2b_E$ are bounded, the terms in $Q_E$ involving $D^2b_E$ contribute only $O_{C_*}(t^{-2k})$ after multiplication by $t^{-k}$. Hence,
		\begin{equation}\label{eq:barrier-source-expansion-simplified}
			\sigma_k(D^2V_t^\sigma)
			=
			f+t^{-k}\big(\sigma C_*+Q_E(D^2w_E^f)\big)+O_{C_*}(t^{-2k}).
		\end{equation} 
		Now, we choose $C_*$ so large that $C_*>2\big\|Q_E(D^2w_E^f)\big\|_{L^\infty(\Omega)}+2$, then we have 
		\begin{equation}\label{eq:global-barrier-ineq}
			\sigma_k(D^2V_t^+)\ge f,\quad \sigma_k(D^2V_t^-)\le f\quad \text{in } \Omega,
		\end{equation}
		for all sufficiently large $t$.
		
		Before applying Lemma~\ref{lem:comparison}, we check that the barriers are admissible. Since $f\ge c_0>0$ and $\Delta_Vb_E=1$,
		\begin{equation}\label{eq:barrier-delta-positive}
			\Delta_V\big(w_E^f+\sigma C_*t^{-k}b_E\big)=f+\sigma C_*t^{-k}\ge \frac{c_0}{2}
		\end{equation}
		for $t$ sufficiently large and for both signs $\sigma=\pm1$. Lemma \ref{lem:admissibility} then gives $D^2V_t^\sigma\in\Gk$ for large $t$. Since $b_E\le0$ in $\Omega$ by Lemma \ref{lem:b-sign}, the definitions give
		\begin{equation}\label{eq:barrier-order}
			V_t^+\le U_{t,E}^f\le V_t^- \quad \text{in } \Omega.
		\end{equation}
		The functions $V_t^+$, $u_{t,E}^f$, and $V_t^-$ have the same boundary value $t\phi_E$ on $\p\Omega$. The exact solution $u_{t,E}^f$ is $k$-admissible by Proposition~\ref{prop:CNS-forward}, and the barriers $V_t^\pm$ are $k$-admissible by Lemma \ref{lem:admissibility}. We now apply Lemma \ref{lem:comparison} twice. First, with $u=V_t^+$ and $v=u_{t,E}^f$, the inequality $\sigma_k(D^2V_t^+)\ge f=\sigma_k(D^2u_{t,E}^f)$ gives $V_t^+\le u_{t,E}^f$ in $\Omega$. Second, with $u=u_{t,E}^f$ and $v=V_t^-$, the inequality $f=\sigma_k(D^2u_{t,E}^f)\ge \sigma_k(D^2V_t^-)$ gives $u_{t,E}^f\le V_t^-$ in $\Omega$. Therefore,
		\begin{equation}\label{eq:u-between}
			V_t^+\le u_{t,E}^f\le V_t^- \quad \text{in } \Omega.
		\end{equation}
		
		We now compare $u_{t,E}^f$ with $U_{t,E}^f$ pointwise. From the definition \eqref{eq:global-barriers} of $V_t^\sigma$, since $b_E\le0$ in $\Omega$, this gives $V_t^+\le U_{t,E}^f\le V_t^-$. Together with \eqref{eq:u-between}, we have $u_{t,E}^f-U_{t,E}^f \le V_t^- - U_{t,E}^f=-C_*t^{1-2k}b_E$ and $U_{t,E}^f-u_{t,E}^f\le U_{t,E}^f - V_t^+=-C_*t^{1-2k}b_E$.  Since $-b_E=|b_E|$, we obtain
		\begin{equation}\label{eq:u-U-bound-by-b}
			\big|u_{t,E}^f-U_{t,E}^f\big|
			\le C_*t^{1-2k}|b_E|
			\quad
			\text{in } \Omega.
		\end{equation}
		The bound on $b_E$ in Lemma \ref{lem:b-sign} gives
		\begin{equation}\label{eq:C0-final-bound}
			\big\|u_{t,E}^f-U_{t,E}^f\big\|_{L^\infty(\Omega)}
			\le
			C_*t^{1-2k}\frac{\diam(\Omega)^2}{2q}.
		\end{equation}
		This proves \eqref{eq:C0-asymp}.
	\end{proof}
	
	The estimate in Proposition \ref{prop:global-C0} is only an $L^\infty$ estimate. It is not enough to differentiate this estimate up to the boundary, since the leading Hessian profile $tP_E$ lies close to a rank $k-1$ degeneracy of the $k$-Hessian operator. Therefore, the usual route via uniform elliptic boundary estimates is unavailable at the leading scale. The boundary normal derivative has to be extracted by comparison. The next lemma constructs local barriers on one-sided boundary neighborhoods of compact subsets of the non-glancing set. These barriers have two roles: their $\Delta_V$ sign creates a source gap larger than the algebraic error in the $\sigma_k$ expansion, and their positive size on $\p D_{x_0}\setminus\p\Omega$ dominates the global $L^\infty$ remainder.
	
	We now prove the boundary normal derivative asymptotic on compact subsets of the non-glancing boundary.
	
	\begin{lemma}\label{lem:uniform-G}
		Let $K\Subset\Gamma_E$, and fix $0<\alpha<1/(4q)$. Then there exist a constant $c>0$ and one-sided boundary neighborhoods $D_{x_0}\subset\Omega$, $x_0\in K$, with smooth boundary and with $x_0\in\p D_{x_0}\cap\p\Omega$, such that the following holds. If
		\begin{equation}\label{eq:Gx0}
			G_{x_0}(x)=-b_E(x)+\alpha |x-x_0|^2,
		\end{equation}
		then $G_{x_0}\ge0$ on $\overline{D_{x_0}}$, $G_{x_0}(x_0)=0$,
		\begin{equation}\label{eq:Delta-G-neg}
			\Delta_VG_{x_0}\le -c \quad \text{in } D_{x_0},
		\end{equation}
		and
		\begin{equation}\label{eq:G-artificial-positive}
			G_{x_0}\ge c \quad \text{on } \p D_{x_0}\setminus\p\Omega.
		\end{equation}
		The constant $c$ and the size of the neighborhoods may be chosen uniformly for $x_0\in K$.
	\end{lemma}
	
	\begin{proof}
		The number $\alpha$ has already been fixed. Since $\Delta_Vb_E=1$ and $\Delta_V|x-x_0|^2=2q$, one has $\Delta_VG_{x_0}=-1+2q\alpha\le -\frac12$. The inequalities $G_{x_0}\ge0$ and $G_{x_0}(x_0)=0$ follow from $b_E\le0$ in $\Omega$ and $b_E=0$ on $\p\Omega$.
		
		It remains to choose the neighborhoods so that $G_{x_0}$ is uniformly positive on $\p D_{x_0}\setminus\p\Omega$. Since $K\Subset\Gamma_E$, Lemma \ref{lem:b-sign} and compactness give a number $c_1>0$ and a relatively open boundary neighborhood $K_1$ with $K\Subset K_1\Subset\Gamma_E$ such that
		\begin{equation}\label{eq:b-normal-uniform}
			\p_\nu b_E\ge c_1 \quad \text{on } K_1.
		\end{equation}
		Using boundary normal coordinates in a fixed one-sided collar of $K_1$, write points as $x=p-s\nu(p)$, where $p\in\p\Omega$ and $0\le s<s_0$. Since $b_E=0$ on $\p\Omega$, Taylor's formula and \eqref{eq:b-normal-uniform} give, after decreasing $s_0$ if necessary,
		\begin{equation}\label{eq:b-linear-collar}
			-b_E(p-s\nu(p))\ge \frac{c_1}{2}s
		\end{equation}
		for $p\in K_1$ and $0\le s<s_0$.
		
		Choose $r>0$ and $s_1>0$ so small that, for every $x_0\in K$, the boundary ball $B_{\p\Omega}(x_0,2r)$ is contained in $K_1$ and the corresponding one-sided collar of height $2s_1$ is contained in the collar above. We choose $D_{x_0}$ to be a smooth one-sided boundary neighborhood satisfying
		\[
		\left\{p-s\nu(p):\, p\in B_{\p\Omega}(x_0,r/2),\ 0<s<s_1/2\right\}\subset D_{x_0}\subset \{p-s\nu(p):\, p\in B_{\p\Omega}(x_0,r),\ 0<s<s_1\},
		\]
		and such that $\p D_{x_0}\cap\p\Omega$ contains $B_{\p\Omega}(x_0,r/2)$. Such domains are obtained by rounding the corners of the rectangular collar in boundary normal coordinates. The construction is uniform for $x_0\in K$.
		
		On $\p D_{x_0}\setminus\p\Omega$, at least one of the following alternatives holds: either $s\ge s_1/2$, or the boundary footpoint $p$ satisfies $\dist_{\p\Omega}(p,x_0)\ge r/2$. In the first case, \eqref{eq:b-linear-collar} gives $G_{x_0}\ge -b_E\ge \frac{c_1s_1}{4}$. In the second case, after decreasing $s_1$ relative to $r$ if necessary, one has $|p-s\nu(p)-x_0|\ge r/4$, and hence, $G_{x_0}\ge \alpha r^2/16$. Taking $c=\min\big\{\frac12,\frac{c_1s_1}{4},\frac{\alpha r^2}{16}\big\}$ gives \eqref{eq:Delta-G-neg} and \eqref{eq:G-artificial-positive}, uniformly for $x_0\in K$.
	\end{proof}
	
	\begin{proposition}\label{prop:normal-asymp}
		Let $K\Subset\Gamma_E$. Then, for fixed $E$,
		\begin{equation}\label{eq:normal-asymp}
			\p_\nu u_{t,E}^f=t\p_\nu\phi_E+t^{1-k}\p_\nu w_E^f+O(t^{1-k-k/2})
		\end{equation}
		uniformly on $K$.
	\end{proposition}
	
	\begin{proof}
		Put $\eps_t=t^{-k/2}$ and $\eta_t=t^{1-k}\eps_t$. The choice $\eps_t=t^{-k/2}$ is only used to separate scales. On one hand, $\eps_t$ is much larger than the algebraic error $t^{-k}$ in the expansion of $\sigma_k$, so the sign of the local source gap is controlled by $\pm\eps_t\Delta_VG_{x_0}$. On the other hand, the resulting boundary displacement $\eta_t=t^{1-k}\eps_t$ is still much larger than the global remainder $t^{1-2k}$ on $\p D_{x_0}\setminus\p\Omega$. Since we have $\frac{t^{1-k-k/2}}{t^{1-2k}}=t^{k/2}\to\infty$, which implies that the local barriers can be compared with the exact solution on the whole boundary of the one-sided neighborhood.
		
		Fix $x_0\in K$ and let $D_{x_0}$ and $G_{x_0}$ be as in Lemma \ref{lem:uniform-G}. Define
		\begin{equation}\label{eq:local-barriers}
			L_{t,x_0}=U_{t,E}^f-\eta_tG_{x_0}, \quad
			W_{t,x_0}=U_{t,E}^f+\eta_tG_{x_0}.
		\end{equation}
		Using \eqref{eq:U-def}, the above identities are equivalent to 
		\begin{equation}\label{eq:local-corrections}
			L_{t,x_0}=t\phi_E+t^{1-k}(w_E^f-\eps_tG_{x_0}), \quad
			W_{t,x_0}=t\phi_E+t^{1-k}(w_E^f+\eps_tG_{x_0}).
		\end{equation}
		By \eqref{eq:Delta-G-neg},
		\begin{equation}\label{eq:local-Delta}
			\Delta_V(w_E^f-\eps_tG_{x_0})\ge f+c\eps_t, \quad
			\Delta_V(w_E^f+\eps_tG_{x_0})\le f-c\eps_t.
		\end{equation}
		The $C^2$ norms of $G_{x_0}$ are bounded uniformly for $x_0\in K$. Applying Lemma~\ref{lem:sigma-expansion} to the two functions in \eqref{eq:local-corrections}, and using that $\eps_t=t^{-k/2}$ dominates the error $t^{-k}$, gives, for all sufficiently large $t$,
		\begin{equation}\label{eq:local-barrier-source}
			\sigma_k(D^2L_{t,x_0})\ge f, \quad \sigma_k(D^2W_{t,x_0})\le f
		\end{equation}
		in $D_{x_0}$. To check admissibility, choose $C_K>0$ such that $|\Delta_VG_{x_0}|\le C_K$ for all $x_0\in K$. Since $f\ge c_0>0$ on $\ol\Omega$, one has
		\begin{equation}\label{eq:local-admissible-delta}
			\Delta_V(w_E^f\pm\eps_tG_{x_0})=f\pm\eps_t\Delta_VG_{x_0}\ge c_0-C_K\eps_t
			\ge \frac{c_0}{2}
		\end{equation}
		for all sufficiently large $t$, uniformly for $x_0\in K$. Lemma~\ref{lem:admissibility} gives that both $L_{t,x_0}$ and $W_{t,x_0}$ are $k$-admissible for large $t$, uniformly for $x_0\in K$.
		
		We compare the barriers with the exact solution on $\p D_{x_0}$. On $\p D_{x_0}\cap\p\Omega$, one has $U_{t,E}^f=t\phi_E=u_{t,E}^f$ and $G_{x_0}\ge0$, so $L_{t,x_0}\le u_{t,E}^f\le W_{t,x_0}$. On $\p D_{x_0}\setminus\p\Omega$, Lemma \ref{lem:uniform-G} gives $G_{x_0}\ge c$. Proposition \ref{prop:global-C0} gives $\big|u_{t,E}^f-U_{t,E}^f\big|\le C_0t^{1-2k}$ in $\Omega$. Since $\eta_tG_{x_0}\ge ct^{1-k-k/2}$ and $1-k-k/2>1-2k$, the ratio $t^{1-k-k/2}/t^{1-2k}=t^{k/2}$ tends to infinity. Thus, $\eta_tG_{x_0}$ is larger than $C_0t^{1-2k}$ on $\p D_{x_0}\setminus\p\Omega$ for all sufficiently large $t$. Hence, the same boundary ordering holds on all of $\p D_{x_0}$.
		
		The domains $D_{x_0}$ were chosen with smooth boundary in Lemma~\ref{lem:uniform-G}. The exact solution is $k$-admissible, and the two local barriers are $k$-admissible on $\ol{D_{x_0}}$ for large $t$. The source inequalities \eqref{eq:local-barrier-source} hold in $D_{x_0}$, and the boundary ordering has just been verified on all of $\p D_{x_0}$. We apply Lemma~\ref{lem:comparison} twice in $D_{x_0}$. First, with $u=L_{t,x_0}$ and $v=u_{t,E}^f$, the inequality $\sigma_k(D^2L_{t,x_0})\ge f=\sigma_k(D^2u_{t,E}^f)$ gives $L_{t,x_0}\le u_{t,E}^f$. Second, with $u=u_{t,E}^f$ and $v=W_{t,x_0}$, the inequality $f=\sigma_k(D^2u_{t,E}^f)\ge\sigma_k(D^2W_{t,x_0})$ gives $u_{t,E}^f\le W_{t,x_0}$. Hence,
		\begin{equation}\label{eq:local-between}
			L_{t,x_0}\le u_{t,E}^f\le W_{t,x_0} \quad \text{in } D_{x_0}.
		\end{equation}
		At $x_0$, all three functions have the same boundary value $t\phi_E(x_0)$. Since $L_{t,x_0}\le u_{t,E}^f\le W_{t,x_0}$ in the one-sided neighborhood and equality holds at $x_0$, the outward normal derivatives satisfy
		\begin{equation}\label{eq:normal-order}
			\p_\nu W_{t,x_0}(x_0)\le \p_\nu u_{t,E}^f(x_0)\le \p_\nu L_{t,x_0}(x_0).
		\end{equation}
		Here, the normal is the outward normal to $\Omega$. Since $D_{x_0}$ is a one-sided neighborhood inside $\Omega$ and $x_0\in\p D_{x_0}\cap\p\Omega$, this is also the outward normal to $D_{x_0}$ at $x_0$.
		
		Using \eqref{eq:local-barriers}, we have
		\begin{equation}\label{eq:barrier-normal-derivatives}
			\p_\nu L_{t,x_0}(x_0)=\p_\nu U_{t,E}^f(x_0)-\eta_t\p_\nu G_{x_0}(x_0), \quad
			\p_\nu W_{t,x_0}(x_0)=\p_\nu U_{t,E}^f(x_0)+\eta_t\p_\nu G_{x_0}(x_0).
		\end{equation}
		Moreover, $G_{x_0}=-b_E+\alpha|x-x_0|^2$, and the normal derivative of $|x-x_0|^2$ at $x_0$ is zero. Therefore,
		\begin{equation}\label{eq:normal-G}
			\p_\nu G_{x_0}(x_0)=-\p_\nu b_E(x_0).
		\end{equation}
		The function $b_E$ is smooth, so $\p_\nu b_E$ is uniformly bounded on $K$. From \eqref{eq:normal-order}--\eqref{eq:normal-G},
		\begin{equation}\label{eq:normal-U-error}
			\big|\p_\nu u_{t,E}^f(x_0)-\p_\nu U_{t,E}^f(x_0)\big|
			\le C_K\eta_t
			=
			C_Kt^{1-k-k/2}.
		\end{equation}
		Since $U_{t,E}^f=t\phi_E+t^{1-k}w_E^f$, this is exactly \eqref{eq:normal-asymp}. The constants are uniform for $x_0\in K$.
	\end{proof}
	
	\begin{corollary}\label{cor:DN-limit}
		For every fixed $E$, 
		\[
		\lim_{t\to\infty}t^{k-1}\big(\Lambda_f(t\phi_E|_{\p\Omega})-t\p_\nu\phi_E\big)
		=\p_\nu w_E^f
		\]
		locally uniformly on $\Gamma_E$, where $w_E^f$ is the smooth fiberwise Poisson correction characterized by \eqref{eq:w_E^f}.
	\end{corollary}
	
	\begin{proof}
		Let $K\Subset\Gamma_E$. Proposition \ref{prop:normal-asymp} gives $t^{k-1}\big(\p_\nu u_{t,E}^f-t\p_\nu\phi_E\big)=\p_\nu w_E^f+O(t^{-k/2})$ uniformly on $K$.
		Since $\Lambda_f(t\phi_E|_{\p\Omega})=\p_\nu u_{t,E}^f\big|_{\p\Omega}$, the desired limit follows on $K$. The compact set $K\Subset\Gamma_E$ was arbitrary.
	\end{proof}

	\section{Recovery of the source}\label{sec:recovery}
	
	We now complete the inverse problem and prove the reconstruction formula stated in the introduction. The only integral geometry input is the elementary Fourier slice identity for the Euclidean affine $q$-plane transform, $1\le q\le n-1$. We include the short Fourier slice argument below to fix the normalization used here; see, for instance, Helgason \cite[Chapter I, \S 6]{Helgason1999} for background on the $d$-plane transform.
	
	The recovery uses only the coefficient of the first lower-order term in the large-data asymptotics. No linearized coefficient is recovered, and no gauge has to be factored out. For a fixed $E$, the DN limit gives the ambient normal derivative of the fiberwise Poisson correction on $\Gamma_E$. Since every nondegenerate section meets the boundary only in $\Gamma_E$, this is enough to compute the flux through the whole section boundary. The section flux is precisely the affine $q$-plane integral of the source. Thus, the restricted large-data DN map determines the affine $q$-plane Radon transform of the zero extension of $f$. The Fourier slice identity gives the reconstruction formula, while its injectivity gives uniqueness.
	
	For $F\in L^1_c(\R^n)$, define its affine $q$-plane transform by
	\begin{equation}\label{eq:q-plane-transform}
		\Rq F(V,y):=\int_{y+V}F\,d\Hq,
	\end{equation}
	where $V\in\Gr(q,n)$ and $y\in V^\perp$.
	
	\begin{lemma}\label{lem:qRadon-injective}
		Let $1\le q\le n-1$ and let $G\in L^1_c(\R^n)$. If $\Rq G(V,y)=0$ for every $V\in\Gr(q,n)$ and for almost every $y\in V^\perp$, then $G=0$ almost everywhere.
	\end{lemma}
	
	\begin{proof}
		We use the Fourier slice argument as follows. We take the Fourier transform with the convention $\widehat G(\xi)=\int_{\R^n}e^{-\mathsf{i}x\cdot\xi}G(x)\,dx$, with $\mathsf{i}=\sqrt{-1}$.
		
		Fix $V\in\Gr(q,n)$. Every $x\in\R^n$ has a unique decomposition $x=y+z$, where $y\in V^\perp$ and $z\in V$. Since $G\in L^1_c(\R^n)$ and 
		\begin{equation}\label{eq:Radon-L1-bound}
			\int_{V^\perp}\big|\Rq G(V,y)\big|\,dy \le \int_{V^\perp}\int_V |G(y+z)|\,d\Hq(z)\,dy=\|G\|_{L^1(\R^n)}<\infty,
		\end{equation}
		then Fubini's theorem gives $\Rq G(V,\cdot)\in L^1(V^\perp)$. Let $\eta\in V^\perp$. Since $\eta$ is orthogonal to $V$, one has $(y+z)\cdot\eta=y\cdot\eta$ for $z\in V$. Therefore,
		\begin{equation}\label{eq:fourier-slice}
			\int_{V^\perp}e^{-\mathsf{i}y\cdot\eta}\Rq G(V,y)\,dy=\int_{\R^n}e^{-\mathsf{i}x\cdot\eta}G(x)\,dx
			=\widehat G(\eta),
		\end{equation}
		which is the Fourier slice identity for the normalization used here.
		
		By assumption, $\Rq G(V,y)=0$ for almost every $y\in V^\perp$. Note that \eqref{eq:fourier-slice} gives $\widehat G(\eta)=0$ for every $\eta\in V^\perp$. Since this holds for every $V\in\Gr(q,n)$, we can test every nonzero frequency. Let $\xi\in\R^n\setminus\{0\}$. Since $\dim\xi^\perp=n-1$ and $q\le n-1$, one can choose a $q$-dimensional plane $V\subset\xi^\perp$, then $\xi\in V^\perp$, and the previous paragraph gives $\widehat G(\xi)=0$.
		
		We have shown that $\widehat G$ vanishes on $\R^n\setminus\{0\}$. Since $G\in L^1(\R^n)$, its Fourier transform is continuous, and therefore also $\widehat G(0)=0$. The uniqueness theorem for the Fourier transform on $L^1(\R^n)$ gives $G=0$ almost everywhere.
	\end{proof}
	
	\begin{proposition}\label{prop:DN-to-Radon}
		Let $f_1,f_2\in C^\infty(\ol\Omega)$ be positive. Suppose that for a fixed $E\in\Gr(k-1,n)$,
		\begin{equation}\label{eq:fixed-E-DN}
			\Lambda_{f_1}(t\phi_E|_{\p\Omega})=\Lambda_{f_2}(t\phi_E|_{\p\Omega})
			\quad \text{on }\p\Omega
		\end{equation}
		for all sufficiently large $t$. Then, with $V=E^\perp$,
		\begin{equation}\label{eq:fixed-V-transform}
			\int_{\Omega\cap(y+V)}(f_1-f_2)\,d\Hq=0
		\end{equation}
		for every $y\in E$.
	\end{proposition}
	
	\begin{proof}
		The equality of the DN maps means that the two normal derivatives agree on $\p\Omega$ for the same large boundary value $t\phi_E|_{\p\Omega}$. Subtracting the common leading term $t\p_\nu\phi_E$ and applying Corollary \ref{cor:DN-limit} gives
		\begin{equation}\label{eq:normal-corrections-equal}
			\p_\nu w_E^{f_1}=\p_\nu w_E^{f_2}
			\quad
			\text{on } \Gamma_E.
		\end{equation}
		Here, the equality is pointwise on $\Gamma_E$, since the convergence in Corollary \ref{cor:DN-limit} is locally uniform there.
		
		Let $y\in\Omega_E$ first. Then the section $\Omega_{E,y}=\Omega\cap(y+V)$ is nondegenerate. By Lemma \ref{lem:section-geometry}, every point of $\p\Omega\cap(y+V)$ belongs to $\Gamma_E$. Since $\p\Omega\cap(y+V)$ is compact and $|P_V\nu|>0$ on it, it is contained in some compact set $K_y\Subset\Gamma_E$. The locally uniform convergence in Corollary \ref{cor:DN-limit} can be applied to the whole section boundary. Applying the flux identity in Lemma \ref{lem:flux} to $f_1$ and $f_2$ gives
		\[
		\int_{\Omega\cap(y+V)}f_j\,d\Hq
		=
		\int_{\p\Omega\cap(y+V)}|P_V\nu|\p_\nu w_E^{f_j}\,dS_{q-1},
		\quad j=1,2.
		\]
		The boundary integrals are equal by \eqref{eq:normal-corrections-equal}. Hence \eqref{eq:fixed-V-transform} holds for every $y\in\Omega_E$.
		
		If $y\notin\ol{\Omega_E}$, then $\Omega\cap(y+V)$ is empty. If $y\in\p\Omega_E$, Lemma \ref{lem:section-geometry} gives that $\ol\Omega\cap(y+V)$ has zero $q$-dimensional Hausdorff measure. In both cases, the integral in \eqref{eq:fixed-V-transform} is zero. This proves the claim for every $y\in E$.
	\end{proof}

	\begin{proof}[Proof of Corollary \ref{cor:reconstruction}]
		By Corollary \ref{cor:DN-limit}, the boundary limit \eqref{eq:reconstruction-NE} exists locally uniformly on $\Gamma_E$ and satisfies
		\begin{equation}\label{eq:NE-equals-normal-w}
			N_E^f=\p_\nu w_E^f
			\quad
			\text{on } \Gamma_E.
		\end{equation}
		Fix $V\in\Gr(q,n)$ and put $E=V^\perp$. If $y\in E$ and $y\in\Omega_E=P_E(\Omega)$, then the section $\Omega\cap(y+V)$ is nondegenerate. By Lemma \ref{lem:section-geometry}, every point of $\p\Omega\cap(y+V)$ lies in $\Gamma_E$. Therefore, using \eqref{eq:NE-equals-normal-w} and the flux identity in Lemma \ref{lem:flux}, we get
		\begin{equation}\label{eq:Sf-equals-Radon}
			\mathcal S_f(V,y)
			=
			\int_{\Omega\cap(y+V)} f\,d\Hq.
		\end{equation}
		If $y\notin\ol{\Omega_E}$, then the section is empty. If $y\in\p\Omega_E$, Lemma~\ref{lem:section-geometry} gives that $\ol\Omega\cap(y+V)$ has zero $q$-dimensional Hausdorff measure. Hence, after extending $f$ by zero outside $\Omega$, \eqref{eq:Sf-equals-Radon} holds for every $y\in E=V^\perp$, with the convention used in the definition of $\mathcal S_f$. In other words,
		\begin{equation}\label{eq:Sf-Radon-transform}
			\mathcal S_f(V,y)=\Rq F(V,y)
			\quad
			\text{for almost every } y\in V^\perp,
		\end{equation}
		where $F=f\mathbf 1_\Omega$.
		
		Let $\xi\ne0$ and choose any $q$-plane $V_\xi\subset\xi^\perp$, then $\xi\in V_\xi^\perp$. Applying the Fourier slice identity \eqref{eq:fourier-slice} to $G=F$ and $V=V_\xi$ gives
		\begin{equation*}
			\widehat F(\xi)
			=
			\int_{V_\xi^\perp}
			e^{-\mathsf{i}y\cdot\xi}
			\Rq F(V_\xi,y)\,dy.
		\end{equation*}
		Using \eqref{eq:Sf-Radon-transform}, this becomes exactly \eqref{eq:reconstruction-Fourier}. The value at $\xi=0$ is determined by continuity of $\widehat F$, and changing the value at one frequency does not affect the inverse Fourier transform.
		
		Finally, the usual Fourier inversion theorem for compactly supported $L^1$ functions gives $F=\mathcal F^{-1}\widehat F$ in $\mathcal S'(\R^n)$. Since $F=f$ in $\Omega$ and $f$ is smooth there, the reconstruction gives $f$ pointwise in $\Omega$. This proves \eqref{eq:reconstruction-inverse-Fourier} and \eqref{eq:reconstruction-pointwise}.
	\end{proof}

	\begin{proof}[Proof of Theorem \ref{thm:main}]
		Let $G:=f_1-f_2$ in $\Omega$, and extend $G$ by zero to $\R^n$, then $G\in L^1_c(\R^n)$. It is enough to show that the affine $q$-plane transform of $G$ vanishes for every $q$-plane and every translate.
		
		Let $V\in\Gr(q,n)$ be arbitrary and set $E=V^\perp$. Since $q=n-k+1$, one has $\dim E=k-1$, so $E\in\Gr(k-1,n)$. Proposition \ref{prop:DN-to-Radon} applied to this $E$ gives
		\begin{equation}\label{eq:Radon-zero-fixed}
			\Rq G(V,y)=0
			\quad
			\text{for every } y\in V^\perp.
		\end{equation}
		Since $V\in\Gr(q,n)$ was arbitrary, the full affine $q$-plane transform of $G$ vanishes.
		
		In the present theorem, $2\le k\le n$, so $1\le q=n-k+1\le n-1$. This is the range covered by Lemma~\ref{lem:qRadon-injective}. Thus, $G=0$ almost everywhere in $\R^n$. Since $f_1-f_2$ is smooth in $\Omega$, it follows that $f_1=f_2$ everywhere in $\Omega$.
	\end{proof}
	
	\begin{proof}[Proof of Corollary \ref{cor:power}]
		All solutions are considered on the $k$-admissible branch. Hence, $\sigma_k(D^2u)>0$, and for positive sources the equation $\sigma_k(D^2u)^\theta=f$ is equivalent to $\sigma_k(D^2u)=f^{1/\theta}$. If the DN maps for the powered equation agree on the same large boundary family, then the DN maps for the $k$-Hessian equation with sources $f_1^{1/\theta}$ and $f_2^{1/\theta}$ agree on that family. Theorem \ref{thm:main} gives $f_1^{1/\theta}=f_2^{1/\theta}$ in $\Omega$, and then $f_1=f_2$ in $\Omega$.
	\end{proof}

	\begin{remark}\label{rem:no-gauge}
		No residual coordinate gauge appears in the present inverse source problem. The proof does not identify a linearized coefficient from a linearized DN map; therefore, it does not enter the anisotropic Calder\'on-type gauge structure. Instead, the nonlinear DN map is evaluated along large rank $k-1$ rays. The first nontrivial large-data coefficient is a scalar fiberwise Poisson correction, and its boundary flux is an invariant section integral of the source. Thus, the inverse step gives direct affine Radon data for $f$, rather than coefficient data modulo a diffeomorphism or conformal factor.
	\end{remark}

	\section*{Statements and declarations}
	
	\para{Data availability statement} No datasets were generated or analyzed during the current study.
	
	\para{Conflict of interest} The author declares no conflict of interest.
	
	\para{Acknowledgments} The author is partially supported by the National Science and Technology Council (NSTC), Taiwan, under the project 113-2115-M-A49-017-MY3. The author also acknowledges financial support from the Alexander von Humboldt Foundation through the Henriette Herz Scouting Programme, hosted by Universit\"at Duisburg-Essen, Germany. 
	
	\bibliographystyle{alpha}
	\bibliography{ref}
	
\end{document}